\documentclass{amsart}

\usepackage{amstext,amssymb,amsmath,amsbsy,mathrsfs,mathtools,amsfonts,amscd,exscale}
\usepackage{graphics, epstopdf}
\usepackage{graphicx}
\usepackage{color}
\usepackage{amsmath, amssymb}
\usepackage{enumerate}
\usepackage{bigints}
\usepackage{verbatim}
\usepackage[poorman]{cleveref}

\usepackage{graphicx}
\usepackage{epstopdf}
\usepackage{subcaption}
\usepackage{float}
\usepackage{placeins}

\usepackage{bm}

\newcommand{\ve}{\varepsilon}

\renewcommand{\sc}{\textsc}

\newcommand{\MA}{Monge-Amp\`{e}re }
\newcommand{\Th}{\mathcal{T}_h}

\newtheorem{Theorem}{Theorem}[section]
\newtheorem{Lemma}[Theorem]{Lemma}

\theoremstyle{definition}
\newtheorem{Definition}[Theorem]{Definition}

\newtheorem{remark}[Theorem]{Remark}

\numberwithin{equation}{section}

\def\l({\left(}
\def\r){\right)}
\def\uve{u_{\varepsilon}}
\def\vve{v_{\varepsilon}}
\def\Uve{\mathbf{U}_{\varepsilon}}
\def\Vve{\mathbf{V}_{\varepsilon}}

\def\Nh{\mathcal{N}_h}
\def\Nhi{\mathcal{N}_h^0}
\def\Nhb{\mathcal{N}_h^b}

\def\sdd{\nabla^2_{\delta_m}} 
\def\sddp{\nabla^{2,+}_{\delta_m}} 
\def\sddm{\nabla^{2,-}_{\delta_m}} 

\def\sdda{\nabla^2_{\delta_a}} 
\def\sddap{\nabla^{2,+}_{\delta_a}} 
\def\sddam{\nabla^{2,-}_{\delta_a}} 

\def\Vperp{\mathbb{S}^{\perp}}
\def\Vperptm{\mathbb{S}^{\perp}_{\theta_m}}
\def\Vperpta{\mathbb{S}^{\perp}_{\theta_a}}
\def\St{\mathbb S_{\theta}}
\def\Stm{\mathbb S_{\theta_m}}
\def\Sta{\mathbb S_{\theta_a}}
\def\Vh{\mathbb{V}_h}
\def\Vth{\mathbb{V}_{2h}}

\def\Tem{T_{\varepsilon,m}}
\def\Tea{T_{\varepsilon,a}}
\def\Tef{T_{\varepsilon,f}}

\def\da{\delta_a}
\def\dm{\delta_m}
\def\ta{\theta_a}
\def\tm{\theta_m}

\def\Cka{C^{2+k,\alpha}}

\def\Wti{W^2_{\infty}}

\def\Bhi{B_i}
\def\interp{\mathcal I_h}

\def\interpo{\interp^1}
\def\interpt{\mathcal{I}_{2h}^2}

\def\ve{\varepsilon}

\def\bv{\mathbf{v}}

\definecolor{red}{rgb}{1,0,0}
\definecolor{blue}{rgb}{0,0,1}

\begin{document}
	
	\title[Convergent filtered scheme for the Monge-Amp\`ere Equation]
	{Convergent two-scale filtered scheme for the Monge-Amp\`ere Equation}
	
	\author{R. H. Nochetto$^1$}
	\address{Department of Mathematics, University of Maryland, College Park, Maryland 20742}
	\email{rhn@math.umd.edu}
	\thanks{$^1$ Partially supported by the NSF Grant DMS -1411808, the
		Institute Henri Poincar\'e (Paris) and the Hausdorff Institute (Bonn).}
	
	\author{D. Ntogkas$^2$}
	\address{Department of Mathematics, University of Maryland, College Park, Maryland 20742}
	\email{dimnt@math.umd.edu}
	\thanks{$^2$ Partially supported by the  NSF Grant DMS -1411808 and
		the 2016-2017 Patrick and Marguerite Sung Fellowship of the
		University of Maryland.}

	\begin{abstract}
		We propose an extension to our monotone and convergent method for the \MA equation in dimension $d \geq2$, that incorporates the idea of filtered schemes. The method combines our original monotone operator with a more accurate non-monotone modification, using an appropriately chosen filter. This results in a remarkable improvement of accuracy, but without sacrificing the convergence to the unique viscosity solution.

	\end{abstract}
	
	\maketitle

	\section{Introduction} \label{S:Introduction}
	We consider the Monge-Amp\`ere equation with Dirichlet boundary condition: 
	\begin{equation} \label{E:MA}
	\left\{
	\begin{aligned}
	\det{D^2u}&=f  & {\rm in} \ &\Omega \subset \mathbb {R}^d,
	\\ u &=g & {\rm on} \ &\partial \Omega,
	\end{aligned}
	\right.
	\end{equation}
	where $\Omega$ is a uniformly
	convex domain and $f \geq 0$ and $g$ are uniformly continuous functions.
        We seek a
	\textit{convex} solution $u$ of \eqref{E:MA}, which is critical for \eqref{E:MA} to be elliptic and 
	have a unique viscosity solution \cite{Gut}. 
	
	The Monge-Amp\`ere equation has a wide spectrum of applications in optimal mass transport problems, geometry, nonlinear elasticity, optics and meteorology. These applications lead to an increasing interest in the investigation of efficient numerical methods. Existing methods for the Monge-Amp\`ere equation include the early work by Oliker and Prussner \cite{OlPr} 
	for space dimension $d=2$, the vanishing moment methods by Feng and Neilan \cite{FeNe1,FeNe2}, the penalty method of Brenner, Gudi, Neilan \cite{BrenNeil}, 
	least squares and augmented Lagrangian methods by Dean and Glowinski \cite{DeGl1, DeGl2, Gl} 
	and the finite difference methods proposed by Froese and Oberman \cite{FrOb1,FrOb2} and Benamou, Collino and Mirebeau \cite{BeCoMi, Mir}. Feng and Jensen \cite{FeJe} have also recently proposed a semi-Lagrangian method that relies on an equivalent Hamilton-Jacobi-Bellman formulation of the Monge-Amp\`ere Equation.
        Schemes in \cite{BeCoMi,FeJe,FrOb1,FrOb2,Mir} are closely related
          to ours and hinge on a wide stencil approach.

        In this work we extend our two-scale method from \cite{NoNtZh,NoNtZh2}, where we use continuous piecewise linear polynomials on a quasi-uniform mesh of size $h$ and an approximation of the determinant that hinges on a second coarser scale $\delta$, in order to solve the \MA equation numerically. In \cite{NoNtZh} we introduce the two-scale method and prove uniform convergence to the viscosity solution of \eqref{E:MA}, whereas in \cite{NoNtZh2} we derive rates of converges in $L^{\infty}$ for classical and viscosity solutions that belong to certain H\"older and Sobolev spaces. The idea of a filtered scheme that we employ here is motivated by the work of Froese and Oberman in \cite{FrOb2}, but follows a different approach; we refer to \cite{ObSal} and \cite{BoFalSa} for stationary and time depentent Hamilton-Jacobi equations. Instead of combining two different methods, we modify our monotone two-scale method into a more accurate, two-scale non-monotone version that still relies on the same variational formulation for the determinant and combine it with the original monotone operator through a filter function. 
	In order to computationally examine the performance of the scheme, we compare the $L^\infty$ error of the monotone, the accurate and the filtered schemes, using the two main examples from \cite{NoNtZh}. We observe that the filtered operator inherits the improved errors from the accurate operator, but allows the monotone operator to dominate the calculations whenever there is a discontinuity of the Hessian. We investigate this behavior and conclude with some computational observations about the scheme. We prove convergence to the viscosity solution of \eqref{E:MA}.
	
	\subsection{Our contribution}
	
	As in \cite{FrOb1,NoNtZh} our method hinges on the following
	formula for the determinant of the positive semi-definite Hessian $D^2w$ of a smooth convex function $w$:
	\begin{equation}\label{E:Det}
	\det{D^2w}(x) = \min\limits_{\boldsymbol{v} \in \Vperp} \prod_{j=1}^d v_j^TD^2w(x) \ v_j ,
	\end{equation}
	where $\Vperp$ is the set of all $d-$orthonormal bases $\boldsymbol{v}=(v_j)_{j=1}^d, \ v_j \in \mathbb{R}^d$. The minimum in \eqref{E:Det} is achieved by the eigenvectors of $D^2w(x)$ and is equal to the product of the respective eigenvalues. We can discretize the above formula in various ways, employing different polynomial spaces and approximations for the directional derivatives given by $v_j^T D^2wv_j$. These choices lead to schemes with different theoretical properties and levels of accuracy. We first briefly recall the discretization used in \cite{NoNtZh,NoNtZh2} and then introduce a more accurate approach. Combining the two leads to the main contribution of this work, which we call the filtered scheme, due to the use of a filter function that allows us to appropriately combine the two discretizations.
	
	\vspace{0.2cm}
	\textbf{Monotone Operator \cite{NoNtZh,NoNtZh2}:}
	We discretize the domain $\Omega$ by a shape regular and
	quasi-uniform mesh $\Th^1$ with spacing $h$, the fine scale, and
	construct a space $\Vh^1$ of continuous piecewise linear
	functions over $\Th^1$. The superscript $1$ of $\Vh^1$ indicates the use of
        linear polynomials whereas that of $\Th^1$ entails the use of straight (affine equivalent) simplices.  We denote by $\Omega_h$ the computational domain, namely the union of the elements. We also denote by $\mathcal{N}_h$ the nodes of
	$\mathcal{T}_h$, and by
        \[
        \mathcal{N}_h^b := \{x_i \in \mathcal{N}_h: x_i \in \partial \Omega_h\},
        \quad
	\mathcal{N}_h^0 := \mathcal{N}_h \setminus \mathcal{N}_h^b
        \]
	the boundary and interior nodes, respectively. We require that $\Nhb \subset \partial \Omega$, which in view of the convexity of $\Omega$ implies that $\Omega_h$ is also convex and $\Omega_h \subset \Omega$. The second and coarser scale, which from now on we call $\dm$ whenever we refer to the monotone operator, is the length of directions we use to approximate second directional derivatives by central second order differences:
\begin{equation} \label{E:2Sc2Dif}
	\sdd  w (x;v) := \frac{ w(x+\dm v) -2w(x) +w(x-\dm v) }{ \dm^2}
	\quad
	\text{and}
	\quad
	|v| = 1 ,
\end{equation}
	for any $w\in C^0(\overline{\Omega})$. Let $\varepsilon = (h,\dm,\tm)$ represent the two scales and a third parameter $\tm$ that is utilized to discretize $\Vperp$ with precision $\tm$. We ask that for any $v$ in the unit sphere  $\mathbb S$, there
	exists $v^{\theta_m}$ that belongs in our discrete approximate set $\Stm$ such that
	$$
	|v-v^{\theta_m}| \leq \theta_m.
	$$
	Likewise, we define the finite set $\Vperptm$: for any
	$\mathbf v^{\tm} = (v_j^{\theta_m})_{j=1}^d \in \Vperptm, \ v_j^{\theta_m} \in \Stm$ and there exists
	$\mathbf v = (v_j)_{j=1}^d \in \Vperp$
	such that $|v_j - v_j^{\theta_m}| \leq \tm $ for all $1 \leq j \leq d$
	and conversely. We can now define the discrete monotone operator to be
	\begin{equation} \label{E:2ScOp}
	T_{\varepsilon,m}[w](x_i) :=\min\limits_{\mathbf{v} \in \Vperptm} \left( \prod_{j=1}^d \sddp w(x_i;v_j) - \sum_{j=1}^d \sddm w(x_i;v_j) \right),
	\end{equation}
	where $\sddp$ and $\sddm$ denote the positive and negative parts of $\sdd$ respectively and $x_i \in \Nhi$. The discrete solution $\uve \in \Vh^1$ satisfies 
	$$
	T_{\varepsilon,m}[\uve](x_i) = f(x_i) \quad \forall x_i \in \Nhi, \quad \uve(x_i) = g(x_i) \quad \forall x_i \in \Nhb.
	$$

	In \cite{NoNtZh} we prove that this discretization of \eqref{E:Det} is monotone and consistent
	and that $\uve$ converges uniformly in $\Omega$ to the unique viscosity solution of \eqref{E:MA}. In \cite{NoNtZh2} we derive rates of convergence in $L^{\infty}(\Omega_h)$ for classical solutions in H\"older and Sobolev spaces and $f>0$ as well as for some special cases of viscosity solutions and $f \geq 0$. Our numerical experiments of \cite{NoNtZh} indicate linear convergence rates, which is rigorously proven in \cite{LiNo} for classical solutions. Therefore, a linear rate is an accuracy barrier for two-scale monotone schemes with piecewise linear elements. A viable way to reduce this error is to increase the polynomial degree, which, given the two-scale nature of our scheme, would require higher order approximation of second directional derivatives. This pair of corrections leads us to introduce what we call the accurate operator.
	
	\vspace{0.2cm}
	\textbf{Accurate Operator:} This time we use quadratic polynomials in order to achieve a better interpolation error and a more accurate discretization of second directional derivatives in order to decrease the truncation error of the operator. To this end, we introduce again two scales $h$ and $\da$, where $\da \ne \dm$ is the coarse scale corresponding to the length of directions used for accurate discretization of second derivatives. We define the space $\Vh^2$ of continuous, piecewise quadratic functions and, in order to maximize the effect of polynomial degree, we employ isoparametric finite elements \cite{BrenScott,Ciarlet,Walker}. We assume that our domain $\Omega$ is piecewise uniformly convex and piecewise $C^{1,1}$, so that we can guarantee the existence of invertible and quadratic maps that transform the master element into elements with curved sides connecting boundary nodes $\Nhb$ \cite{Ciarlet}. We call the resulting mesh $\mathcal{T}_{h}^2$, the superscript indicating quadratic isoparametric mappings for boundary elements. We also employ a more accurate approximation of the second directional derivatives that relies on five, rather than
	three, point stencils. Consequently, second differences for $\uve \in \Vh^2$ are now given by
	\begin{equation}\label{E:SecDifAc}
\footnotesize
	\sdda  \uve (x_i;v) := \frac{ -\uve(x_i+\da v) + 16\uve(x_i+\frac{\da}{2} v)-30 \uve(x_i) 
		+16\uve(x_i-\frac{\da}{2} v) -\uve(x_i-\da v)  }{ 3\da^2},
	\end{equation}
	\normalsize
	where $x_i \in \Nhi$ and $v \in \Sta$. The symbol $\Sta$ indicates that we use a different angle discretization parameter $\ta$ for the accurate operator. 
	The accurate scheme then becomes: We seek $\uve \in \Vh^2$ such that $\uve(x_i)=g(x_i)$ for
	$x_i \in \mathcal{N}_h^b$ and for $x_i \in \mathcal{N}_h^0$
	\begin{equation} \label{E:AccOp}
	\Tea[\uve](x_i):=\min\limits_{\mathbf{v} \in \Vperpta} \left( \prod_{j=1}^d \sddap \uve(x_i;v_j) - \sum_{j=1}^d \sddam \uve (x_i;v_j) \right) = f(x_i),
	\end{equation}
	We observe that this discretization is no longer monotone, since a change in $\uve$ at a certain $x_j$ could affect the second difference at another node $x_i$ in two possible ways. It can either decrease or increase it, depending on whether it affects the behavior of $-\uve(x _i\pm\da v)$ or $\uve(x_i\pm \frac{\da}{2}v)$, respectively. We also note that $\Tea$ is defined on a space of piecewise quadratic functions, which means that the behavior at nodes does not translate monotonically to the behavior inside simplices. As a result a method relying only on this discretization cannot be proven to converge to viscosity solutions of \eqref{E:MA}.
	This is the motivation behind the use of a filtered scheme, along the lines of \cite{FrOb2}.

	\vspace{0.2cm}
	\textbf{Filtered Scheme:} The idea is to use of a filter function that combines the accurate and 
	the monotone operator and guarantees that the monotone operator will be used if the accurate 
	operator fails, due to the lack of monotonicity. This allows for a notion of an ``almost monotone" 
	operator that is flexible enough to deliver better accuracy for each fixed mesh size. We introduce the scheme here briefly and expand on its theoretical properties later.

        	\FloatBarrier
\begin{figure}[!htb]
	\includegraphics[scale=0.25]{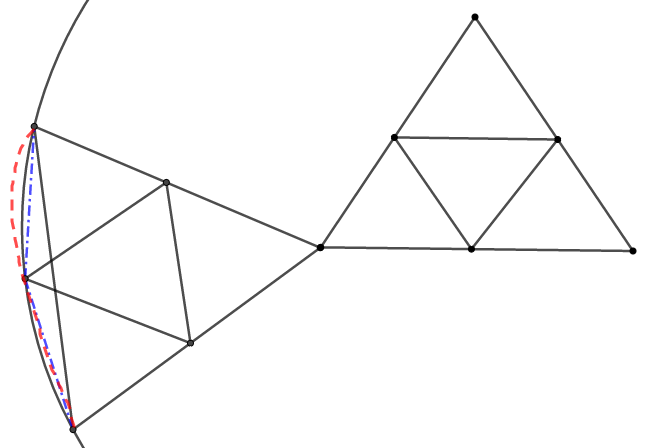}
	\caption{\small Illustration of the refinement close and away from the boundary for $d=2$. In the interior we create four new triangles upon each refinement, by connecting the midpoints of the coarser cell. The space $\Vh^1$ corresponding to $\mathcal{T}_h^1$ and defined over these four triangles shares the same nodal values as the space $\Vth^2$ corresponding to $\mathcal{T}_{2h}^2$ and defined over the original triangle. On boundary elements, the boundary point that is used in the construction and as a nodal value of the isoparametric element in $\mathcal{T}_{2h}^2$ for $\Vth^2$ becomes the new boundary node for $\Omega_h^1 \supset \Omega_{2h}^1$ that is used for the four new triangles of $\mathcal{T}_h^1$. The blue ``$-\cdot-\cdot-$'' dashed line corresponds to the new edges introduced after the refinement. The red ``$- - -$" dashed line illustrates the curved edge of the curved element that is isoparametric to the original triangle.
		\label{F:Domain} }
\end{figure}
\FloatBarrier

We start with the two meshes and function spaces used. Let $\mathcal{T}_h^1$ be a shape regular and quasi-uniform mesh of size $h$ and $\Vh^1$ be the corresponding space of continuous piecewise linear elements. Let $\mathcal{T}_{2h}^2$ be an isoparametric mesh of size $2h$ with same nodes as $\mathcal{T}_h^1$ and $\Vth^2$ be the corresponding space of continuous isoparametric piecewise quadratic elements; see Figure \ref{F:Domain} for $d=2$. An important consequence of this two-grid approach is that functions in $\Vh^1$ and $\Vh^2$ have degrees of freedom at the same nodes, including mid-points on the curvilinear boundary of $\Omega_{2h}^2$.

We now exploit this structure as follows. Let $\Uve=\left\{U_\ve^i\right\}_i\in\mathbb{R}^N$ be a grid function where $N$ is the number of nodes of either $\Vh^1$ or $\Vth^2$. We define two functions $\uve^1 \in \Vh^1$ and $\uve^2 \in \Vth^2$ with nodal values dictated by $\Uve$
\[
\uve^1 (x_i) = \uve^2 (x_i) = U_\ve^i
\quad\forall \, x_i\in\Nh,
\]
and compare them via a filter function $F$; see \Cref{S:FilterChoice} for an explicit definition. We thus seek $\Uve\in\mathbb{R}^N$ such that $U_\ve^i =g(x_i)$ for all $x_i \in \Nhb$ and for all $x_i \in \Nhi$
\begin{equation} \label{E:FiltOp}
\Tef [\Uve](x_i):= \Tem[\uve^1](x_i) + \tau F\left( \frac{\Tea[\uve^2](x_i)-\Tem[\uve^1](x_i)}{\tau} \right) = f(x_i).
\end{equation}
The filter function $F$ is required to be compactly supported and continuous (hence uniformly bounded), as well as equal to the identity close to the origin. Therefore, the difference of operators $\Tea[\uve^2](x_i)-\Tem[\uve^1](x_i)$ relative to the filter scale $\tau= \tau( \ve)$ is a decisive factor for the performance of the scheme. We observe that $\tau( \ve)$ depends on the scales $\ve = (h,\da,\dm,\ta,\tm)$ of the accurate and monotone operators and, since they in turn depend ultimately on $h$, it is important to realize that $\tau \to 0$ as $h \to 0$. We later provide some insight, based on heuristics and experimental evidence, on how to choose $\tau$. We now emphasize here the two main properties of $F$ that we wish to exploit:

\begin{enumerate}[$\bullet$]
	
\item \textit{Minimize the risk:} The accurate operator \eqref{E:AccOp} exhibits a smaller consistency error than the monotone operator \eqref{E:2ScOp} (see \Cref{L:FullConsistencyMon} (consistency of $\Tem[\interpo u]$) and \Cref{L:FullConsistencyAc} (consistency of $\Tea[\interpt u]$)), but the improved accuracy comes at the cost of lack of monotonicity.  Therefore, given the importance of monotonicity for convergence to viscosity solutions, the solution of \eqref{E:AccOp} cannot be guaranteed to converge. This is especially relevant when the right-hand side $f$ degenerates and a classical solution of \eqref{E:MA} might not exist. Since lack of monotonicity could in principle lead to the failure of the accurate operator, the filter $F$ examines the quantity $\Tea[\uve^2](x_i)-\Tem[\uve^1](x_i)$ relative to the filter scale $\tau$. If this diference is smaller than $\tau$, thus signaling that $\Tea[\uve^2](x_i)$ is well behaved, then $F$ is the identity and $\Tef [\Uve](x_i)=\Tea[\uve^2](x_i)$ as desired. If instead this difference is larger than $\tau$, thereby indicating erratic behavior of $\Tea[\uve^2](x_i)$, then $F$ vanishes and $\Tef [\Uve](x_i)=\Tem[\uve^1](x_i)$ signifies that the monotone operator dominates and yields convergence. We note, however, that this is a rather simplistic approach that could lead to using the monotone operator even in cases where the accurate operator is much better. We explore and discuss this further in \Cref{S:FilteredNumerics}.

\item \textit{Almost monotonicity:} A key feature of $F$ is uniform boundedness, namely $|F(s)| \leq 1$ for all $s \in \mathbb{R}$. This leads to \Cref{L:AlmostMon} (almost monotonicity), which, in turn, is critical to prove existence of a solution of \eqref{E:FiltOp} and its convergence to the visicosity solution of \eqref{E:MA}. Proving such results is an essential component of this work, which entails suitable definitions of the filter $F$ and of the filter operator $T_{\ve,f}$ in \eqref{E:FiltOp}. We discuss this in \Cref{S:FilterChoice} including the possible degeneracy of the right-hand side $f$. We provide explicit definitions of $F$.
\end{enumerate}

In order to explore the performance of the accurate operator and the filtered scheme, a substantial part of our presentation is devoted to numerical experiments. We first verify computationally the increased accuracy of the higher-order operator and provide some computational remarks. We then repeat our experiments for the filtered scheme and obtain the anticipated results: the scheme has a smaller error compared to the monotone operator and appears to detect singularities when the solution is not smooth. In fact, we observe that in the case of singularities the filtered scheme performs even better than the accurate operator. We explore this behavior and examine the interplay between the monotone and the accurate operator in the non-smooth case, in order to elucidate the role of filter $F$. Since this is a significant component of this work, we present the numerical examples first in Sections \ref{S:AccurateNumerics} and \ref{S:FilteredNumerics}. We conclude with a discussion of consistency of the monotone and accurate operators in Section \ref{S:Consistency}, as well as proofs of existence and convergence of solutions to the filtered scheme in Section \ref{S:FilteredScheme}.

\section{Definition of Filter Function} \label{S:FilterChoice}

We start with explicit definitions for the filter function $F$, including potential degeneracy of the right-hand side $f$. We recall that $\dm\ne\da$ are the coarse scales in the definition of the operators $T_{\ve,m}$ and $T_{\ve,a}$. We introduce the simplifying notation
			\begin{equation} \label{E:Argument}
			A[\Uve](x_i) := \Tea[\uve^2](x_i)-\Tem[\uve^1](x_i).
			\end{equation}
for the argument of $F$ in \eqref{E:FiltOp}. 
		\FloatBarrier
We choose the following continuous and uniformly bounded filter function
		\begin{equation}\label{E:Filter}
		F_\sigma(s):= 
		\begin{cases*}
		s,  \quad |s| \leq 1  \\
		0,  \quad |s| \geq 1+ \sigma \\
		-\frac{1}{\sigma} \ s + \frac{1+\sigma}{\sigma}, \quad 1 < s <1+\sigma \\
		-\frac{1}{\sigma} \ s - \frac{1+\sigma}{\sigma}, \quad -1-\sigma < s < -1.
		\end{cases*}
		\end{equation}
The parameter $\sigma$ encodes a smooth transition of $F_\sigma$ to zero; its choice and use are further discussed in \Cref{S:FilteredNumerics}. Function $F_\sigma$ satisfies the desirable properties mentioned in Section \ref{S:Introduction}, i.e. it is uniformly bounded by one and coincides with the identity in the interval $[-1,1]$.

However, one important feature of $\Tem$ reported in \cite{NoNtZh} is that it can guarantee the discrete convexity of the discrete solution to $T_{\ve,m}[u_\ve^1](x_i)=f(x_i)$, i.e.
\[
\sdd \uve^1(x_i;v_j) \geq 0
\quad\forall \, x_i \in \Nhi, ~v_j \in \Stm,
\]
provided $f\ge0$. This may not be the case with \eqref{E:Filter} when the right-hand side $f$ touches zero. Although it is possible to deal with this issue asymptotically, it is also desirable to mimic the properties of the continuous problem, which justifies preserving the discrete convexity of discrete solutions.

We explain now why \eqref{E:Filter} may not guarantee discrete convexity and suggest a way to enforce it.
We observe that for every grid function $\mathbf{U}_h$ with corresponding functions $u_h^1 \in \Vh^1$ and $u_h^2 \in \Vth^2$ with nodal values dictated by $\mathbf{U}_h$, and for all $x_i \in \Nhi$ there exists $|k_i|\leq 1$ depending on $\Tea[u_h^2](x_i)$ and $\Tem[u_h^1](x_i)$ such that
		\begin{equation}\label{E:ThreeCasesTf}
		\Tef[\Uve] (x_i) =
		\begin{cases}
		\Tem[u_h^1](x_i), &  |A[\mathbf{U}_h](x_i)| \geq (1+\sigma)\tau \\
		\Tem[u_h^1](x_i) - |k_i| \tau, &  -(1+\sigma)\tau < A[\mathbf{U}_h](x_i) \leq 0 \\
		\Tem[u_h^1](x_i) + |k_i| \tau, &  0 < A[\mathbf{U}_h](x_i) < (1+\sigma)\tau.
		\end{cases}
		\end{equation}
Suppose that $f(x_i)=0$ for some $x_i \in \Nhi$. If we want $\uve^1$ to be discretely convex at $x_i$, we need to make sure that $\Tem[\uve^1](x_i) \geq 0$. This property is also instrumental in \Cref{L:Existence} (existence and stability) to prove existence of a discrete solution using results from \cite{NoNtZh}. The above calculation shows that
		\begin{equation}\label{E:ThreeCasesTm}
		\Tem[\uve^1](x_i) = 
		\begin{cases}
		f(x_i) =0, &  |A[\Uve](x_i)| \geq (1+\sigma) \tau \\
		f(x_i) + |k_i| \tau \geq 0, & -(1+\sigma)\tau < A[\Uve](x_i) \leq 0 \\
		f(x_i) - |k_i| \tau \leq 0, &  0 < A[\Uve](x_i)(x_i) < (1+\sigma)\tau.
		\end{cases}
		\end{equation}
This reveals that in order to preserve $\Tem[\uve^1](x_i) \geq 0$ for all $x_i \in \Nhi$, we must exclude the third case. We thus introduce a non-symmetric modification of the filter function:
	\begin{equation} \label{E:NSFilter}
	\widetilde F_\sigma (s) := 
	\begin{cases}
	s, & -1 \leq s \leq 0 \\
	-\frac{1}{\sigma} \ s - \frac{1+\sigma}{\sigma}, &-1-\sigma \leq s \leq -1 \\
	0, & \text{otherwise.}
	\end{cases}
	\end{equation}

        From now on we make the convention that \eqref{E:Filter} is used in \eqref{E:FiltOp} whenever the right-hand side $f(x)\ge f_0>0$ for all $x\in\Omega$ whereas \eqref{E:NSFilter} is our choice provided $f$ touches zero. We emphasize that this decision depends on $f$ but not on the space location $x$. In both cases, we have
          \begin{equation}\label{E:discrete-convex}
          \Tem[\uve^1](x_i) \ge f(x_i) - |k_i|\tau \ge 0,
          \end{equation}
provided $\tau\le f_0$ in the first case and by construction in the degenerate case. Consequently, $\uve^1$ is discretely convex.

          In \Cref{L:Existence} (existence and stability) we show the existence of a discretely convex solution of \eqref{E:FiltOp}. The restriction $\tau \leq f_0$ is not stringent in practice because $\tau$ is some positive power of $h$ and thus tends to zero. On the other hand, choosing the non-symmetric filter $\widetilde F_\sigma$ destroys the symmetry of the resulting system and excludes parts of the domain where $\Tea$ could have still been used, i.e. whenever $f \geq \tau$. Consequently, we only employ \eqref{E:NSFilter} if necessary. In \Cref{S:FilteredNumerics} we test both \eqref{E:Filter} and  \eqref{E:NSFilter} on a smooth example with strictly positive $f$ and a $C^{1,1}$ example with vanishing $f$, respectively. We also explore briefly a space-dependent choice of filter for the degenerate case.

\section{Numerical Experiments: Accurate Scheme} \label{S:AccurateNumerics}
In this section we illustrate the improved performance of the accurate operator.

  \subsection{Comparison between $T_{\ve,m}$ and $T_{\ve,a}$} \label{S:mono-vs-acc}
We present in \Cref{Ta:Accurate} the $L^\infty$ error of the solution of \eqref{E:2ScOp} vs that of \eqref{E:AccOp} for the following two examples taken from \cite{NoNtZh} and defined on $\Omega = [0,1]^2$:
	
	\vskip0.2cm\noindent
	{\bf Smooth Hessian:} Let the exact solution $u$ and forcing $f$ be
	\[
	u(x) = e^{|x|^2/2},
	\quad
	f(x) = (1+|x|^2) e^{|x|^2}
	\quad\forall x\in\Omega.
	\]
	\medskip\noindent
	{\bf Discontinuous Hessian:} Let $x_0 = (0.5,0.5)$ and $u$ and $f$ be
	\[
	u(x) =  \frac{1}{2} \left(\max({|x-x_0|-0.2,0)}\right)^2, 
	\quad
	f(x) =\max{\left(1-\frac{0.2}{|x-x_0|},0\right)}
	\quad\forall x\in\Omega.
	\]

        Since $\Omega$ is polygonal, the computational domain $\Omega_h=\Omega$ and the isoparametric maps of $\mathcal{T}_{2h}^2$ for boundary elements are simply affine. This choice simplifies the numerics and allows us to compare with earlier experiments from \cite{NoNtZh}. We indeed compare $\|u-\uve^1\|_{L^\infty(\Omega)}$ and $\|u-\uve^2\|_{L^\infty(\Omega)}$ where $\uve^1\in\Vh^1$ solves \eqref{E:2ScOp} and $\uve^2\in\mathbb{V}_{2h}^2$ solves \eqref{E:AccOp}, whence the number of degrees of freedom is the same in both examples.

        For the smooth case, we observe that the accurate operator exhibits a significant improvement beyond one order of magnitude. Although there is no theoretical result to support this fact, it can be formally explained by the regularity of the solution and the higher order operator consistency error in \Cref{L:FullConsistencyAc} (consistency of $\Tea[\interpt u]$). It is worth noting that the accuracy improvement for the monotone operator for the smooth example exhibits saturation in the last refinement. Upon examining where the error is larger, we realize that it appears on the boundary layer that arises from the definition of $\Tem$. As shown in \cite[Theorem 5.3]{NoNtZh2}, this error does not obey operator consistency, but is instead bounded by $C \ \|u\|_{W^2_\infty(\Omega)} \ \dm$ through a barrier argument. 
        
	\FloatBarrier
	\begin{table}[tbh]
		\begin{center}
			\begin{tabular}[t]{ | c | c | c | c | c |}
				\hline
				DoFs    & $P$: \ \# of points & $\Tem$ & $ \Tea$   & Newton steps \\
				\hline\hline
				N= 4225, $h=2^{-6}$	 &	  56   &   $2.8	 \ 10^{-3}$   	&  $1.91 \ 10^{-4}$			& 6  \\ 
				\hline 
				N=16641,  $h=2^{-7}$	  &    88 & $1.5 \ 10^{-3}$ 	 & 	$8.60 \ 10^{-5}$			&  5 \\
				\hline
				N=66049,  $h= 2^{-8}$	&	144	  &      $7.8 \ 10^{-4}$    & 	$4.17 \ 10^{-5}$			&  5 \\
				\hline
				N= 263169,  $h=2^{-9} $  	&	224	   &   $6.4 \ 10^{-4}$      & 	$2.42  \ 10^{-5}$			&   7  \\ 
				\hline
			\end{tabular}
		\end{center}
		\begin{center}
			\begin{tabular}[t]{ | c | c | c | c | c |}
				\hline
				DoFs      & $P$: \ \# of points & $\Tem$ & $ \Tea$ & Newton steps \\
				\hline\hline
				N= 4225, $h=2^{-6}$	    &	     40  &  $1.9 \ 10^{-3}$   &  $2.48 \ 10^{-4}$			& 9  \\ 
				\hline
				N=16641,  $h=2^{-7}$	    &	     56  &  $9.0 \ 10^{-4}$    & 	$1.51 \ 10^{-4}$			& 12 \\
				\hline
				N=66049,  $h= 2^{-8}$  &		72	 &     $5.7 \ 10^{-4}$       & 	$ 8.34 \ 10^{-5}$			&  14	 \\
				\hline
				N= 263169,  $h=2^{-9} $  &		96    &  $3.83 \ 10^{-4}$      & 	$5.31 \ 10^{-5}$	&   20 \\ 
				\hline
			\end{tabular}
		\end{center}
		\caption{ \footnotesize $L^{\infty}$ error for monotone and accurate operators, $T_{\ve,m}$ and $T_{\ve,a}$, for the smooth Hessian (top) and the discontinuous Hessian (bottom).
			$P:$ number of points $x_i\pm\delta_a v_j,x_i\pm\frac{\delta_a}{2} v_j$
			used for $\Tea$ at each $x_i\in\Nhi$; $P=8(D-1)$, 
			where $D$ is the number of directions $v_j$
			in a quarter circle, as determined by the value of $\theta_a$.
			The operator $T_{\ve,a}$ has an accuracy of one to two orders higher than $T_{\ve,m}$ \cite{NoNtZh}.
			The number of Newton iterations corresponds to $T_{\ve,a}$.}
		\label{Ta:Accurate}
	\end{table}
 
For the example with discontinuous Hessian we observe again an accuracy improvement from  $\Tem$ to $\Tea$, despite the fact that the predicted error in \cite[Theorem 5.7]{NoNtZh2} for $T_{\ve,m}$ and a degenerate $f$ is determined by the dimension of the problem rather than the regularity of the solution. We also notice, in contrast to \cite{NoNtZh}, an increase in the number of Newton iterations with each refinement for $T_{\ve,a}$. This may be attributed to the lack of monotonicity of $\Tea$. 

\subsection{Computational Remarks} \label{S:AcCompRe}
We now explain implementation issues for the accurate method, which are in turn relevant for the filtered scheme.

\smallskip
\textbf{Sparsity:} The evaluation of $\sdda\uve^2(x_i;v_j)$ using \eqref{E:SecDifAc} requires about twice the number of points as $\sdd\uve^1(x_i;v_j)$ using \eqref{E:2Sc2Dif}, for each point $x_i \in \Nhi$ and direction $v_j \in \mathbb{S}_\theta$. This results in a sparsity pattern with a wider bandwidth and more non-zero elements, since for each extra point $x_i \pm \frac{\delta}{2}v_j$, we need to use the degrees of freedom of the simplex where $x_i\pm \frac{\delta}{2}v_j$ belong.

	
\smallskip	
	\textbf{Solver:} The monotone operator $T_{\ve,m}$ in \cite{NoNtZh} was implemented using a direct solver for the linear system, i.e. Matlab's backslash operator. However, for very fine meshes which yield many directions an iterative method like conjugate gradient may be more appropriate. The situation is more critical for the accurate operator $T_{\ve,a}$ due to its worse sparsity pattern discussed above. This makes a direct solver a less favorable option for very fine meshes. Choosing a small tolerance for the conjugate gradient method seems to provide computational results similar to the direct solver, thus without sacrificing accuracy but gaining efficiency. The design of suitable preconditioners is essential, but remains an open issue.

	\section{Numerical Experiments: Filtered Scheme} \label{S:FilteredNumerics}
	
	We now explore computationally the accuracy of the filtered scheme in \eqref{E:FiltOp}, which combines the monotone and accurate operators. In fact, we determine the \textit{active set} of the filtered scheme, which is the region where the monotone operator dominates. We start with a brief discussion about the choice of $\sigma$ in $F_\sigma$ and fix a value for our implementation.
	
	\subsection{Choice of $\sigma$ in $F_\sigma$:}
We observe that the function $F_\sigma$ in \eqref{E:Filter} is Lipschitz for any $\sigma>0$ whereas for $\sigma = 0$ is the discontinuous function 
	$$
	F_0(s):= 
	\begin{cases*}
	s,  \quad |s| \leq 1  \\
	0,  \quad |s| >1.
	\end{cases*}
	$$
        Since continuity of $F_\sigma$ is only used in \Cref{L:Existence} (existence and stability) in order to apply \cite[Lemma 3.1]{NoNtZh}, we can choose $\sigma$ as small as we want and then perform computations with decreasing values of $h$. In practice, we take $\sigma = 10^{-4}$, which leads to such a tiny window for the last two cases in \eqref{E:Filter} that they do not occur in practice. This is desirable because our goal is to give full control to $\Tea$ in regions of smoothness and to employ the monotone operator $T_{\ve,m}$ otherwise.

        The use of semi-smooth Newton for the range $1\le|s|\le1+\sigma$ is however questionable. For example, if $\sigma = 1$ and $-2 \leq \tau^{-1}\big(\Tea[\uve^1](x_i)-\Tem[\uve^2](x_i) \big) \leq -1$ in \eqref{E:Filter}, we obtain 
	$$
	\begin{aligned}
	\Tef[\Uve](x_i) = 2 \ \Tem[\uve^1](x_i) -\Tea[\uve^2](x_i) - 2 \tau.
	\end{aligned}
	$$
	This would result in the $i$-th row $\nabla \Tef[\Uve](x_i)$ of the Jacobian matrix to be
		$$
		\nabla \Tef[\Uve](x_i) = 2 \ \nabla \Tem[\uve^1](x_i) - \nabla \Tea [\uve^2](x_i),
		$$
		where $\nabla \Tem[\uve^1]$ and $\nabla \Tea[\uve^2]$ are the Jacobian matrices associated with the monotone and accurate operators. Both Jacobians behave well computationally, but there is no reason to expect that a matrix resulting from subtracting them will be non-singular. In fact, we observe computationally that the semi-smooth Newton becomes very slow and for very fine meshes it does not even converge. Concerns about the solvability of the Newton system are also raised  in \cite{FrOb2}, where the following approximation is advocated
	$$
	\nabla \Tef[\Uve](x_i) \approx 2 \ \nabla \Tem[\uve^1](x_i).
	$$

	We prefer, instead, to avoid these issues altogether by choosing a rather small value of $\sigma$.  For any fixed value of $\sigma$, we still employ a semi-smooth Newton iteration and treat all the corners of the filter $F_\sigma$ similarly to the min and max functions in \cite{NoNtZh}. Moreover, using $\sigma = 10^{-4}$ the likelihood of $|\Tea[\uve^2](x_i)-\Tem[\uve^1](x_i)| \in [1,1+\sigma]$
	is rather small and indeed it rarely occurs in practice. We see in \Cref{Ta:FilteredExp,Ta:FilteredC11} in section \ref{S:numer-exp} that our choice leads computationally to a similar amount of Newton iterations as for the accurate scheme in \Cref{Ta:Accurate}.

	\vspace{0.1cm}

	\subsection{Numerical Experiments}\label{S:numer-exp}
	Before presenting our results in detail, using the same examples as in \Cref{S:AccurateNumerics}, we make two general observations.
	\begin{enumerate}[$\bullet$]
	\item \textbf{Choice of filter scale $\tau$:} We stress that the behavior of the scheme depends strongly on the 	choice of the filtered scale $\tau$, which must obey $\tau\to0$ as $h\to0$. A bigger $\tau$ allows the accurate operator to take control, while the monotone operator guarantees convergence when the accurate operator has a very large consistency error. On the other hand, smaller values of $\tau$ lead to the presence of the active set of nodes, where the monotone operator dominates. Since we measure the error in the $L^{\infty}$ norm, a large active set could prevent us from achieving better accuracy than in \cite{NoNtZh}. Although there is no obvious recipe for choosing $\tau$, the definition \eqref{E:FiltOp} of $\Tef$ indicates that, in order for the accurate operator to be active at a point $x_i$, we need $|\Tea[\uve^2](x_i)-\Tem[\uve^1](x_i)|\leq \tau$. Consequently, $\tau$ has to be greater than the truncation error of the monotone operator, because the active set may otherwise include nodes where the solution is smooth. We follow this approach to generate \Cref{Ta:FilteredExp,Ta:FilteredC11} and \Cref{F:Errors_Comparison}.

        \smallskip
        \item \textbf{Boundary Layer:} Our experiments reveal that the active set may contain nodes near $\partial \Omega$. This is due to the different boundary layer effect of each operator. In fact, the consistency error for the monotone operator is of order one because $\frac{\delta_m^2}{h^2}\approx 1$ (see \Cref{L:FullConsistencyMon} (consistency of $\Tem[\interpo u]$)), while for the accurate operator the order becomes $\frac{\delta_a^2}{h^2} \approx \frac{h^3}{h^2} \approx h$ (see \Cref{L:FullConsistencyAc} (consistency of $\Tea[\interpt u]$)).
	\end{enumerate}
        
	We now document the performance of the filtered scheme for the two examples of \Cref{S:AccurateNumerics} and investigate the effect of $\tau$ in the size and location of the active set. We compare the performance of the scheme with that of the monotone and accurate operators, and recall that $\Omega_{h}=\Omega$ for all $h>0$.
 
	\smallskip\noindent
	{\bf Experiment 1:  Smooth Hessian.}
	We start by illustrating the error estimates for the smooth example for $\tau = 6 e^2h$ in \Cref{Ta:FilteredExp}. This choice is motivated by the theoretical truncation error of $\Tem$ for the corresponding choice of $\dm$. For this example we use the original, symmetric, filter $F_\sigma$ with $\sigma = 10^{-4}$. This falls under the existence and convergence results of \Cref{S:FilteredScheme}, since $f(x)= (1+|x|^2)e^{|x|^2} \geq e > \tau$ in $\Omega$ for all values of $h$ that are used in \Cref{Ta:FilteredExp}. We observe a small active set, with relative size around $1\%-2\%$ for all refinements.
	\FloatBarrier
	\begin{table}[h!]
		\begin{center}
			\begin{tabular}[t]{ | c | c | c | c | c | c |}
				\hline
				$h$     & $\Tem$ & $\Tea$ & $\Tef$  & $\Tef$ Newton& Active Set \\
				\hline\hline
				$h=2^{-5}$	    &	 $5.4 \ 10^{-3}$   & $5.16 \ 10^{-4}$   &  $1.01 \ 10^{-3}$			& 6   & 17  \\ 
				\hline
				$h=2^{-6}$	    &	  $2.8	 \ 10^{-3}$   	&  $1.91 \ 10^{-4}$	     & 	$3.16 \ 10^{-4}$			& 6 & 57 \\
				\hline
				$h= 2^{-7}$  &	 $1.5 \ 10^{-3}$ 	 & 	$8.60 \ 10^{-5}$    & 	$ 1.20 \ 10^{-4}$			&  6 & 214	 \\
				\hline
				$h=2^{-8} $  &	 $7.8 \ 10^{-4}$    & 	$4.17 \ 10^{-5}$	 & 	$5.00 \ 10^{-5}$			&   7 & 918  \\ 
				\hline
				$h=2^{-9} $  & $6.4 \ 10^{-4}$      & 	$2.42  \ 10^{-5}$		 & 	$1.99 \ 10^{-5}$			&   8& 5035  \\ 
				\hline
			\end{tabular}
		\end{center}
		\FloatBarrier
		\vskip0.2cm
		\caption{\small Smooth Hessian, $\tau = 6e^2h$. The $L^\infty$ error for $\Tef$ is up to an order of magnitude smaller than $\Tem$. Active Set is the number of nodes where the low-accuracy operator $\Tem$ dominates. The amount of Newton iterations for the filtered scheme is also displayed.} 
		\label{Ta:FilteredExp}
	\end{table}
	\FloatBarrier
	\normalsize
	The active set for $\tau = 6 e^2h$ and $h = 2^{-8}$ is displayed in \Cref{F:ActiveSetExp}. We observe the aforementioned boundary layer effect, especially close to the upper-right corner, where $u$ and its derivatives are larger.

	\FloatBarrier
	\begin{figure}[!htb]
		\includegraphics[scale=0.27]{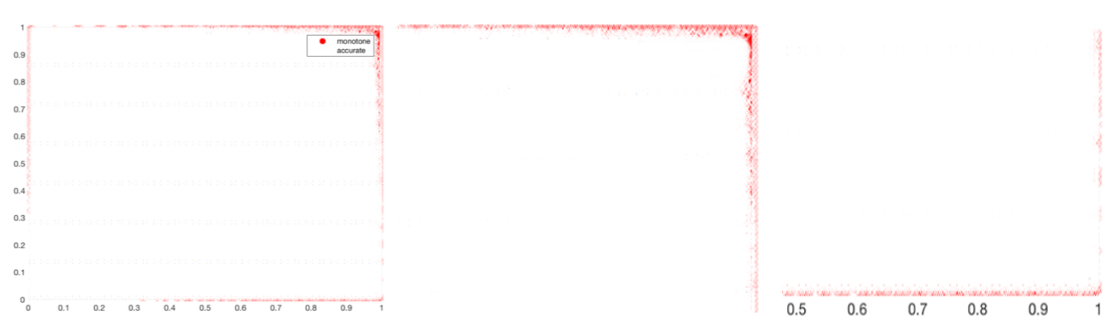}
		\caption{\small Active set for smooth example:  $\tau = 6 e^2  h$. The active set is concentrated on the upper corner and, less, on the lower and left side. We illustrate the upper corner and lower side in a zoom-in. }
		\label{F:ActiveSetExp}
	\end{figure}
	\FloatBarrier
	
	\smallskip\noindent
	{\bf Experiment 2: Discontinuous Hessian.}
	For the $C^{1,1}$ example of \Cref{S:mono-vs-acc} we choose $\tau = 0.62 \  h^{2/5}$, which is motivated by the theoretical consistency error of the monotone operator for $\dm = 2.44 \ h$. We use the non-symmetric filter $\widetilde F_\sigma$ with $\sigma=10^{-4}$ in \eqref{E:NSFilter}. We observe that the active set is located on the circle of discontinuity of the Hessian and near the boundary. We also see that the error in the $L^{\infty}$ norm is slightly better than in \Cref{Ta:Accurate}. We observe computationally that coarser choices of $\tau$, as for example $\tau = O(h^{1/5})$, lead to an empty active-set. In contrast to the smooth example, we now notice a gradual increase of the relative size of the active set. This is more prominent in the last three refinements, where this relative size increases from around $1.5\%$ to $2.9\%$ and then to $4.1\%$.
	\FloatBarrier
	\begin{table}[tbh]
		\begin{center}
			\begin{tabular}[t]{ | c | c | c |  c | c | c |}
				\hline
				$h$     & $\Tem$ & $\Tea$ & $\Tef$  & $\Tef$ Newton& Active Set \\
				\hline\hline
				$h=2^{-5}$	    &	    $4.0 \ 10^{-3}$ &    $5.67 \ 10^{-4}$   &  $5.50 \ 10^{-4}$			& 5   & 22   \\ 
				\hline
				$h=2^{-6}$	    &	 $1.9 \ 10^{-3}$   &  $2.48 \ 10^{-4}$   & 	$2.48 \ 10^{-4}$			& 8 & 8 \\
				\hline
				$h= 2^{-7}$  &	$9.0 \ 10^{-4}$    & 	$1.51 \ 10^{-4}$   & 	$ 1.40 \ 10^{-4}$			&  12 &  248	 \\
				\hline
				$h=2^{-8} $  &	  $5.7 \ 10^{-4}$       & 	$ 8.34 \ 10^{-5}$   & 	$7.58 \ 10^{-5}$			&   14 & 1904   \\ 
				\hline
				$h=2^{-9} $  &	$3.83 \ 10^{-4}$      & 	$5.31 \ 10^{-5}$   & 	$4.89 \ 10^{-5}$			&   16&   10825 \\ 
				\hline
			\end{tabular}
		\end{center}
		\FloatBarrier
		\vskip0.2cm
		\caption{\small Discontinuous Hessian, $\tau = 0.62 \ h^{2/5}$. The $L^\infty$ error for $\Tef$ is about one order of magnitude better than that of $\Tem$ and slightly better than of $\Tea$. The latter is explained in \Cref{F:Contour}.}
		\label{Ta:FilteredC11}
	\end{table}
	\FloatBarrier
	
	In order to explain the smaller errors of \Cref{Ta:FilteredC11} with respect to \Cref{Ta:Accurate}, we present in \Cref{F:Contour} a contour plot with $|\uve-u|$ for the accurate scheme and meshsize $2h=2^{-6}$. In the same figure we depict the active set associated with the filtered scheme, and observe that the circle of discontinuity of the Hessian dominates the active set and is precisely the set of nodes where the accurate operator exhibits the biggest error. This explains why using the monotone operator at these points increases the accuracy and provides experimental justification of the filtered scheme.

\FloatBarrier
\begin{figure}[!htb]
	\begin{center}
		\includegraphics[scale=0.25]{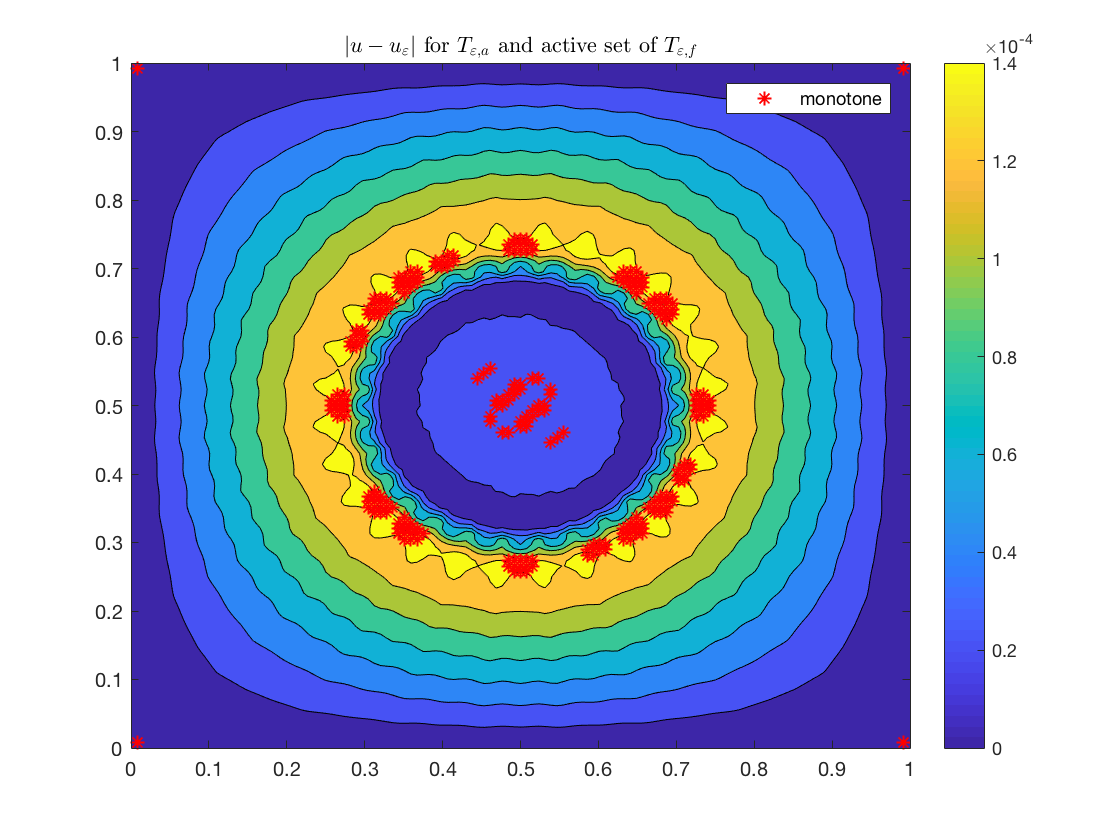}
	\end{center}
		\caption{\small Error $|\uve-u|$ of $\Tea$ for meshsize $2h=2^{-6}$ and active set for filtered scheme (in red). Most nodes on the circle of discontinuity of the Hessian, as well as the corners of $\Omega$, belong to the active set and are the nodes $x_i$ with the highest absolute error for $\Tea[\uve^2](x_i)$.}
		\label{F:Contour}
	\end{figure}
	\FloatBarrier

	We now explore further the behavior of the active set. We illustrate it on \Cref{F:ActiveSetC11} (left) for the final iteration that corresponds to $h=2^{-6}$. We observe that it includes the layer $\dm-$ away from the circle of discontinuity, the center of the domain, where $f=0$ and the problem degenerates, and a small boundary layer at the corners of the domain.
	\FloatBarrier
	\begin{figure}[!htb]
		\includegraphics[scale=0.43]{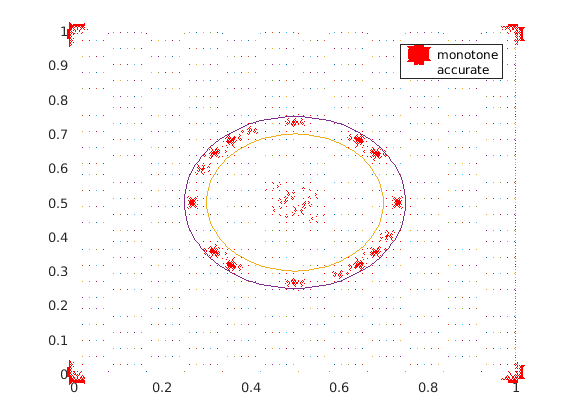}
		\includegraphics[scale=0.17]{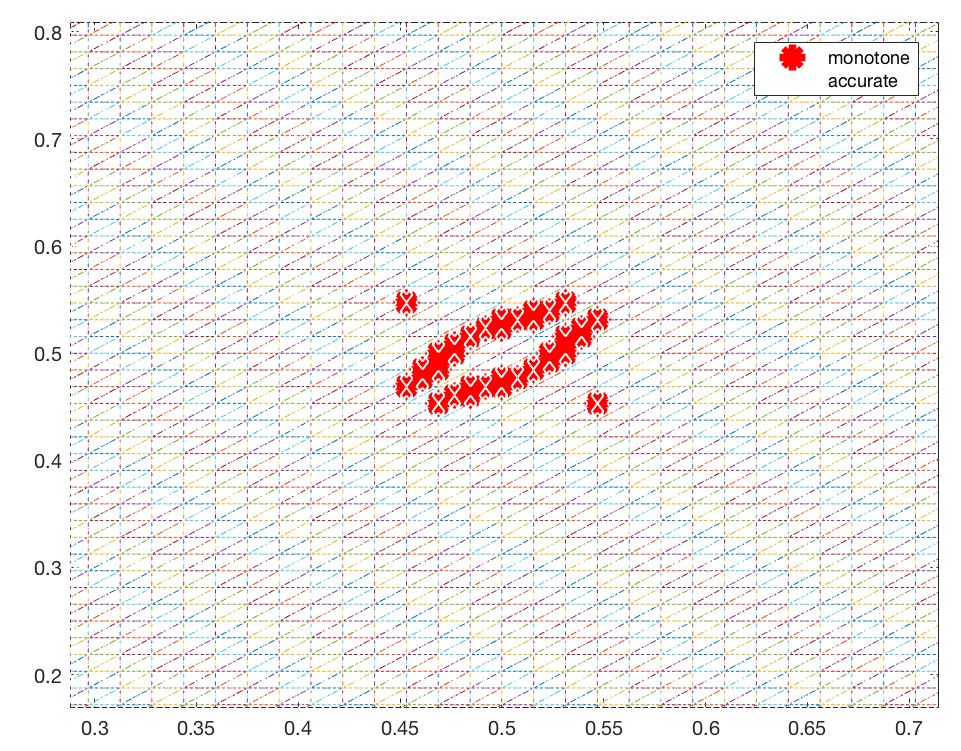}
		\caption{\small (a) Active set for discontinuous Hessian: $h =2^{-6}$, $\tau = 0.62 \ h^{2/5 } $ and $\widetilde F_\sigma$ with $\sigma=10^{-4}$ (left). The two circles correspond to $\left\{ |x-x_0|=0.2\right\}$ and $\left\{ |x-x_0| = 0.2+ \dm \right\}$. (b) Active set for a space-dependent filter function (right).}
		\label{F:ActiveSetC11}
	\end{figure}
	\FloatBarrier
	That the active set reduces to a layer around the circle of discontinuity of the Hessian is a desirable property of the filtered scheme, whereas nodes near the corners are caused by the disparate consistency errors of the operators near the boundary. In order to investigate the presence of active nodes within $\left\{ |x-x_0| \leq 0.2 \right\}$, we experiment with a space-dependent filter function: we restrict the use of the accurate operator only for those nodes $x_i \in \Nhi$ such that $f(x_i) \leq \tau$. This function helps to shed some light on the behavior of the two operators. We present in \Cref{F:ActiveSetC11}(b) a zoomed-in illustration of the active set that corresponds to the same parameters as in \Cref{F:ActiveSetC11}(a), but using this space-dependent filter. We observe that the active set is now only restricted to the center of the circle. By examining the values of the two operators in the active set, we realize that the monotone operator is of order $10^{-18}$, while the accurate operator is of order $10^{-8}$, which explains why this active set remains present. Since $f =0$ in $\left\{ |x-x_0| \leq 0.2 \right\}$, we want to choose the operator whose value is closer to zero. This is why both \eqref{E:NSFilter} and its space-dependent version enforce the monotone operator whenever the accurate operator is ``positive enough" in this region. For \eqref{E:NSFilter} this includes parts of the Hessian discontinuity at the center of the circle, while for the space-dependent version of the filter, this is only enforced in the center of the domain. We see that, although at first sight the definition of \eqref{E:NSFilter} may raise concerns about restricting the accurate operator too much, it actually allows us to employ the monotone operator precisely at the critical nodes. This is further portrayed in \Cref{F:Contour} which illustrates that the active set due to \eqref{E:NSFilter} includes nodes of lowest accuracy of $\Tea$. This is why \Cref{F:Errors_Comparison} displays smaller errors for the filtered operator than for the accurate one.
	
	\begin{figure}[!htb]
		\includegraphics[scale=0.15]{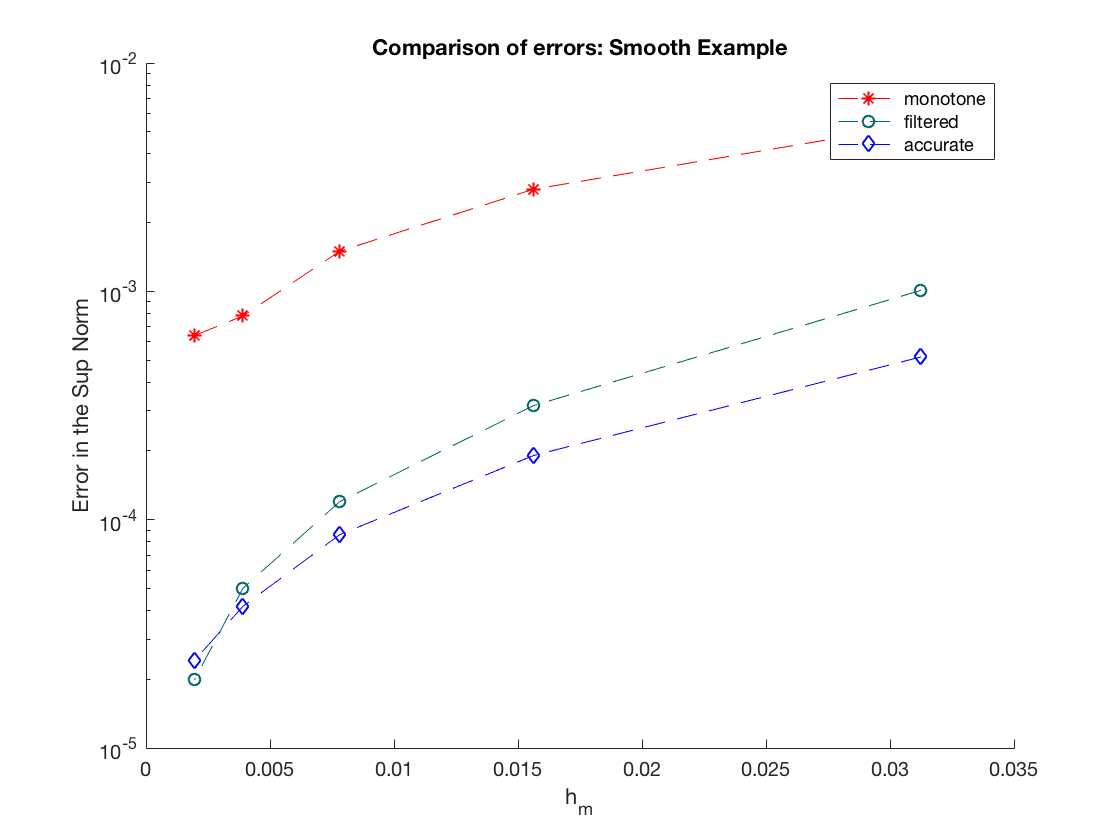}
		\includegraphics[scale=0.15]{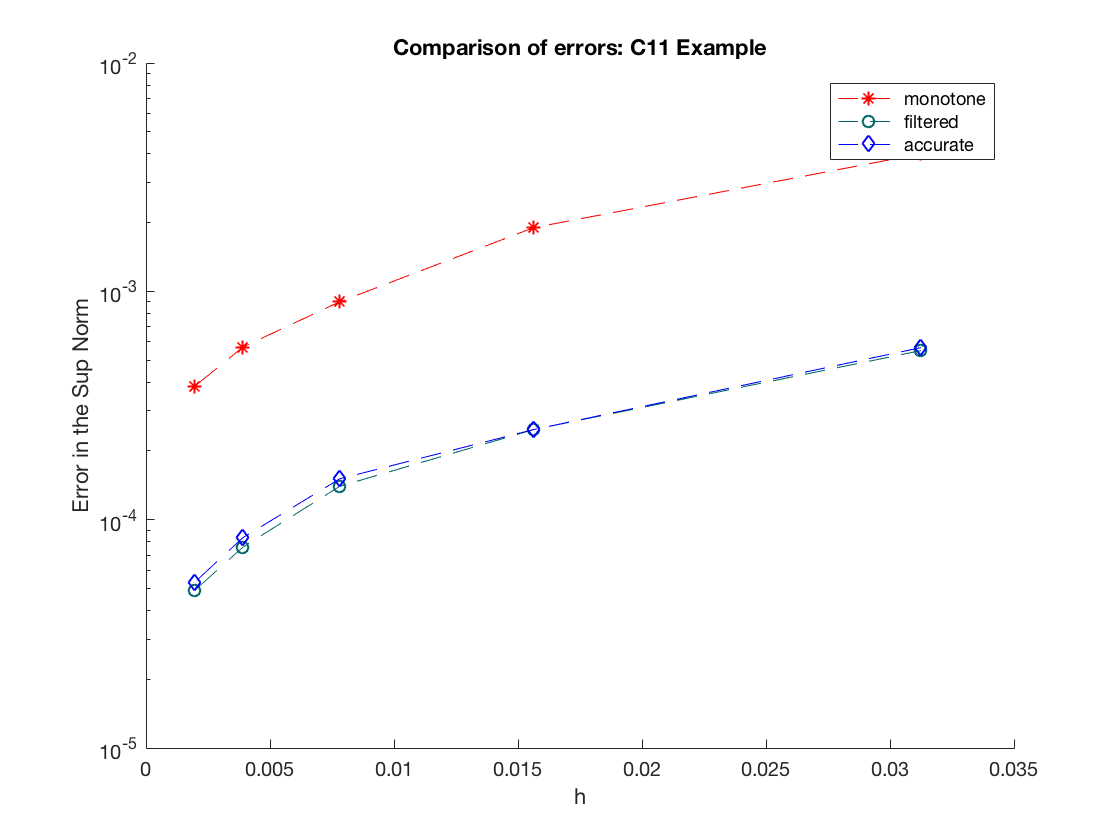}
		\caption{\small Error for monotone, accurate and filtered operator: a) Smooth example, for filtered scheme: $\tau = 4e^2 h $
			(left), b) $C^{1,1}$ example, for filtered scheme:  $\tau = 0.41 \ h^{2/5 } $	(right).}
		\label{F:Errors_Comparison}
	\end{figure}

        \smallskip\noindent
	\textbf{Conclusions:}
	We see that in both examples the choice of the filtered scale $\tau$ is not a trivial task and it can have a dramatic effect on the outcome. Since the accurate operator performs better close to the boundary, it is expected that for many choices of $\tau$ there will be a boundary layer, where the monotone operator will be active. Unfortunately, this is due to the bad behavior of the monotone operator near the boundary and not to the filter capturing a singularity. This can be explained by the simplicity of the filter functions being used, because they compare only the values of the two operators and do not take into account the respective value of the right hand side.
 
	A crucial observation is that the convergence result allows for a great deal of flexibility in the choice of $\tau$. We can always choose $\tau$ to be relatively big with respect to $\varepsilon$, but still satisfy $\tau \to 0$ as $\varepsilon \to 0$. This results in the accurate operator being the one that is always active, hence allowing us to fully exploit the higher accuracy it offers, without sacrificing the convergence result. We do not present a table for this case, because it corresponds exactly to the results of \Cref{S:AccurateNumerics}. Instead we present two comparative figures for the errors due to the monotone, the accurate and the filtered scheme, which correspond to the results from \Cref{Ta:FilteredExp} and \Cref{Ta:FilteredC11}. We see that the filtered scheme is much more efficient than the monotone one even in the presence of boundary layers, and outperforms the accurate scheme in the non-smooth case.
	
	\subsection{Choice of Solver}
	The discussion in \Cref{S:AcCompRe} indicates that we need to take into account the sparsity of the resulting Jacobian matrix and possibly choose between a direct and an iterative solver. This choice depends on the problem at hand. For instance, for a strictly positive right hand side $f\ge f_0>0$, we may choose $\dm$ and $\tm$ as well as $\da$ and $\ta$ (with some modifications) on the basis of \cite[Theorem 5.3 (rates of convergence for classical solutions)]{NoNtZh2}. On the other hand, a right hand side $f$ that touches zero may yield a choice of scales as described in \cite[Theorem 5.6 (degenerate forcing $f \geq 0$)]{NoNtZh2}. Note that the ensuing constants are not accessible in either case. The first choice corresponds to a smaller angular parameter and leads to more directions, hence to a larger bandwidth for the Jacobian matrix. If in addition $f$ is sufficiently smooth so that the solution $u$ is $C^2(\Omega)$ and strictly convex, then the Jacobian matrix is likely to be positive definite which, combined with the reduced sparsity, makes it preferable to use an iterative solver such as the conjugate gradient. In contrast, a degenerate $f\ge0$ that touches zero does not guaranteed strict convexity and global regularity of $u$. In this case, \cite[Theorem 5.6 (degenerate forcing $f \geq 0$)]{NoNtZh2} suggests a coarser choice of $\tm$ which leads to a sparser Jacobian matrix that makes a direct solver competitive.

	\subsection{Implementation Challenges:}
	The above exposition shows that the filtered scheme provides the desirable combination of provable convergence and increased accuracy and adds to the family of similar approaches, such as \cite{FrOb2}. However, this improvement is not without challenges. In particular, implementing the filtered scheme requires special care in the following three aspects.
	\begin{enumerate} [$\bullet$]
		\item {\it Second differences:} Similarly to \cite{NoNtZh}, for each node $x_i \in \Nhi$ we need to locate the appropriate simplex where $x_i \pm \frac{\delta}{2}v_j$ and $x_i \pm \delta v_j$ belong to, in order to calculate the second differences for the monotone and accurate operator. This is a process that now needs to take place for both operators. We employ the efficient searching techniques of FELICITY to achieve this in minimal time, as in \cite{NoNtZh}.
		\item {\it Initialization:} To construct the initial guess, we solve a Laplace problem on the coarsest mesh and, for each subsequent refinement, we interpolate the solution of \eqref{E:FiltOp} on the previous mesh. This is an efficient choice from \cite{NoNtZh}, suggested earlier in \cite{FrOb2}, that we preserve for both the accurate and filtered schemes. To achieve this, we need to interpolate a quadratic function on a finer mesh. We thus create the new mesh $\Th^2$ at the end of each iteration, and use the searching and interpolation capabilities of FELICITY, to associate each node of $\Th^2$ with a simplex of $\mathcal{T}_{2h}^2$ and interpolate $\uve^2$ in $\Th^2$.
		\item {\it Two-mesh approach:} The definition \eqref{E:FiltOp} of filter operator utilizes a piecewise linear function $\uve^1 \in \Vh^1$ and a piecewise quadratic function $\uve^2 \in \Vth^2$, each defined on a different mesh $\mathcal{T}_{h}^1$ and $\mathcal{T}_{2h}^2$, and each involving the same number, $N$, of degrees of freedom. Even though $\mathcal{T}_{h}^1$ and $\mathcal{T}_{2h}^2$ are compatible, the global numbering of nodes, and thus of degrees of freedom, is in practice often different.  Consequently, in order to compare $\Tem[\uve^1](x_i)$ and $\Tea[\uve^2](x_i)$ at each node $x_i$, we need to communicate between $\mathcal{T}_{h}^1$ and $\mathcal{T}_{2h}^2$ and their degrees of freedom. However, exploiting the efficiency of Matlab for vectorized quantities, creating a map between the degrees of freedom for the two meshes takes minimal time. 
	\end{enumerate}

	\section{Properties of Monotone and Accurate Operators}\label{S:Consistency}
	We now embark on the theoretical analysis of our method. To this end, we first briefly recall key properties of the monotone operator of \cite{NoNtZh} and next present the notion of consistency of both the monotone and accurate operators and compare them. These properties are important to prove the existence and convergence results for the filtered scheme. One of the critical properties of the \MA equation is the convexity of its solution $u$. We mimic this property at the discrete level, using the notion of 	discrete convexity \cite[Definition 2.1]{NoNtZh}.
	\begin{Definition}[discrete convexity]\label{D:discrete-convexity}
		We say that $w_h \in \mathbb V_h^1$ is discretely convex  if
		$$
		\sdd w_h(x_i;v_j) \geq 0 \qquad \forall x_i \in \Nhi, \quad
		\forall v_j \in \St.
		$$
	\end{Definition}
	We have the following lemma for the monotone operator $\Tem$ \cite[Lemma 2.2]{NoNtZh}.
	\begin{Lemma}[discrete convexity of $\Tem$]\label{L:DisConv}	
		If  $w_h \in \Vh^1$ satisfies
		\begin{equation} \label{E:Oper}
		\Tem[w_h](x_i) \geq 0 \quad \forall x_i \in \mathcal{N}_h^0,
		\end{equation}
		then $w_h$ is \textit{discretely convex} and as a consequence
		\begin{equation}\label{E:simpler-def}
		\Tem [w_h](x_i)= \min_{\mathbf{v} \in \Vperptm} \prod_{j=1}^d \sdd w_h(x_i;v_j),
		\end{equation}
		namely 
		$$
		\sddp w_h(x_i;v_j) = \sdd w_h(x_i;v_j),
		\quad
		\sddm w_h(x_i;v_j) =0
		\quad\forall x_i\in \mathcal{N}_h^0,\quad\forall v_j\in \Stm.
		$$
		Conversely, if $w_h$ is discretely convex, then (\ref{E:Oper}) is valid.
	\end{Lemma}
	A critical feature for the convergence of $\Tem[\uve^1]$ is monotonicity \cite[Lemma 2.3]{NoNtZh}.
	\begin{Lemma}[monotonicity of $\Tem$] \label{L:Monotonicity} 
		Let $u_h,w_h \in \Vh^1$ be discretely convex. If $u_h - w_h$ attains a 
		maximum at an interior node $z \in \Nhi$, then 
		$$
		\Tem [w_h](z) \geq \Tem [u_h](z).
		$$
	\end{Lemma}
	Another important property that relies on monotonicity is the following 
	discrete comparison principle \cite[Lemma 2.4]{NoNtZh}.
	\begin{Lemma}[discrete comparison principle for $\Tem$] \label{L:DCP}
		Let $u_h,w_h \in \Vh^1$ with $u_h \leq w_h$ on the boundary $\partial \Omega_h$ be such that 
		\begin{equation}\label{E:comparison}
		\Tem[u_h](x_i) \geq \Tem[w_h](x_i)  \geq 0 \ \ \ \forall x_i \in \Nhi. 
		\end{equation}
		Then, $u_h \leq w_h$ everywhere.
	\end{Lemma}
	We now provide a consistency estimate for $\Tem$ and compare it later with $\Tea$. To this end, given a node $x_i \in \Nhi$ we denote by
	\begin{equation}\label{E:Bi}
	B_i := \cup \{\overline{T}: T\in\Th, \, \textrm{dist }(x_i,T) \le \hat\delta\}
	\end{equation}
	where $\hat\delta := \rho \delta$ with $0<\rho\le 1$ is so that
	$x_i\pm\hat\delta v_j\in\overline{\Omega}_h$ for all $v_j\in\St$. 
	We also introduce the $\delta$-interior region
	$$
	\Omega_{h,\delta} = \left\{ T \in \mathcal{T}_h \ : \  {\rm dist}
	(x,\partial \Omega_h) \geq \delta  \ \forall x \in T  \right\}.
	$$
	The above notation is introduced for general $\delta$ and $\theta$ and is adjusted accordingly for each of the following consistency lemmas. The following result is proven in \cite[Lemma 4.2]{NoNtZh}.
	\begin{Lemma}[{consistency of $\Tem [\interpo u]$}] \label{L:FullConsistencyMon}
		Let $x_i \in \Nhi \cap\Omega_{h,\dm}$ and $B_i$ be defined as in
		\eqref{E:Bi}. Let also $u \in C^{2+k,\alpha}(B_i)$ with $0<\alpha\leq 1$ and
		$k=0,1$ convex, and let $\interpo u $ be its piecewise linear interpolant. 
		Then 
		\begin{equation}\label{E:FullConsistency}
		\left|  \det D^2u(x_i) - \Tem[\interpo u] (x_i)  \right| \leq C_1(d,\Omega,u) \dm^{k+\alpha} + C_2(d,\Omega,u) \left( \frac{h^2}{\dm^2} + \tm^2 \right),
		\end{equation}
		where
		\begin{equation}\label{E:C1-C2}
		C_1(d,\Omega,u)= C |u|_{C^{2+k,\alpha}(B_i)} |u|_{W^2_{\infty}(B_i)}^{d-1}, \quad C_2 = C |u|_{W^2_{\infty}(B_i)}^d.
		\end{equation}
		If $x_i\in\Nhi$ and $u \in W^2_{\infty}(B_i)$, then (\ref{E:FullConsistency}) remains valid with $\alpha=k=0$ and $\Cka(\Bhi)$ replaced by $\Wti(\Bhi)$.
	\end{Lemma}
	
	\begin{remark}[optimal consistency error] \label{R:OptConsMon} We observe that, for a smooth function $u$, \Cref{L:FullConsistencyMon} gives a consistency error of order $h$ upon equating $\dm^2, \tm^2$ and $\frac{h^2}{\dm^2}$ and taking 
		$\dm= h^{1/2}$; this explains the accuracy barrier alluded to in \Cref{S:Introduction}. 
	\end{remark}

	We now present a consistency lemma for $\Tea$ to showcase the improved formal accuracy of $\Tea$ relative to $\Tem$. We only sketch the proof of the result, since it mostly follows along the lines of \cite[Lemma 4.1 (consistency of $\sdd \interp u$)]{NoNtZh}. Note that the use of isoparametric finite elements in $\mathcal{T}_{2h}^2$ does not affect the following results, because in the interior of the domain the elements of $\mathcal{T}_{2h}^2$ are straight and the $\da$-interior domain $\Omega_{h,\da}$ does not contain curved boundary cells.

        \begin{Lemma} [consistency of $\sdda \interpt u $] \label{L:ConsistencyAc}
		Let $u \in W^3_\infty(B_i)$, $\interpt u$ be its Lagrange
		interpolant in $\Vth^2$, and $B_i$ be defined in \eqref{E:Bi}.
		The following two estimates are then valid:
		\begin{enumerate}[(i)]
			\item 
			For all $x_i  \in \Nhi$ and all  $v_j \in \Sta$, we have
			$$
			\left|\sdda \interpt u (x_i;v_j) \right|\leq C \ |u|_{W^2_{\infty}(B_i) } .
			$$
			\item
			If in addition $u\in C^{2+k,\alpha}(B_i)$ for $k=0,\ldots,3$ and $\alpha \in (0,1]$,
			then for all $x_i \in \Nhi \cap \Omega_{h,\da}$ and all $v_j \in \Sta$, 
			we have
			$$
			\left|\sdda \interpt u (x_i;v_j) - \frac{\partial^2 u}{\partial v_j^2}(x_i) \right| \leq C \left(|u|_{C^{2+k,\alpha}(B_i)} \da^{k+\alpha} + |u|_{W^3_{\infty}(B_i)} \frac{h^3}{\da^2} \right).
			$$
		\end{enumerate}
		In both cases $C$ stands for a constant independent of the two scales
		$h$ and $\da$, the parameter $\ta$ and the function $u$.	
	\end{Lemma}
	\begin{proof}
We rewrite \eqref{E:SecDifAc} as
			$$
			\sdda u(x;v) = \frac{4}{3} \frac{u(x+\frac{\da}{2})-2u(x)+u(x-\frac{\da}{2})}{\left(\frac{\da}{2}\right)^2} - \frac{1}{3}
			\frac{u(x+\da v)-2u(x)+u(x-\da v)}{\da^2}
			$$
			and then proceed as in \cite[Lemma 4.1 (consistency of $\sdd \interp u$)]{NoNtZh} to obtain 
			$$
			\sdda u(x_i;v) \leq C |u|_{W^2_{\infty}(B_i)}.
			$$
			We then incorporate the $L^\infty$-interpolation error estimate for quadratics \cite{BrenScott}
			$$
			\|u - \interpt u\|_{L^\infty(\Omega_h)} \leq C \ |u|_{W^m_{\infty}(B_i)} h^m \quad m=2,3,
			$$
			with $m=2$, along with the relation $h \leq \da$, to complete the proof of (i). To prove (ii) we make use of $m=3$ and exploit cancellation of higher order derivatives built in the definition of $\sdda u(x_i;v)$. We refer to \cite[Lemma 4.1(ii)]{NoNtZh} for similar estimates for $\Tem$.
	\end{proof}
        
	\Cref{L:ConsistencyAc} implies the following result, which is similar to \cite[Lemma 4.2]{NoNtZh}.
	\begin{Lemma}[{consistency of $\Tea [\interpt u]$}] \label{L:FullConsistencyAc}
		Let $x_i \in \Nhi \cap\Omega_{h,\da}$ and $B_i$ be defined as in
		\eqref{E:Bi}. If $u \in C^{2+k,\alpha}(B_i)$ with $0<\alpha\leq 1$ and
		$k=0,\ldots,3$ is convex, and $\interpt u $ is its piecewise quadratic interpolant, then 
		\begin{equation}\label{E:FullConsistencyAc}
                \begin{aligned}
		  \left|  \det D^2u(x_i) - \Tea[\interpt u] (x_i)  \right| &\leq C_1(d,\Omega,u) \da^{k+\alpha}
                  \\
                  & + C_2(d,\Omega,u)  \frac{h^3}{\da^2} + C_3(d,\Omega,u) \ta^2 ,
                \end{aligned}    
		\end{equation}
		where
		\begin{equation}\label{E:C1-C2A}
		\begin{aligned}
		  & C_1(d,\Omega,u)= C |u|_{C^{2+k,\alpha}(B_i)} |u|_{W^2_{\infty}(B_i)}^{d-1},
                  \\
                  & C_2 = C |u|_{W^3_{\infty}(B_i)}|u|_{W^2_{\infty}(B_i)}^{d-1},
                  \\
		  & C_3 = C |u|_{W^2_{\infty}(B_i)}^d.
		\end{aligned}
		\end{equation}
		If $x_i\in\Nhi$ and $u \in W^2_{\infty}(B_i)$, then (\ref{E:FullConsistencyAc}) remains valid with $\alpha=k=0$ and $\Cka(\Bhi)$ replaced by $\Wti(\Bhi)$.
	\end{Lemma}
	\begin{proof}
		We argue as in \cite[Lemma 4.2]{NoNtZh}, namely we use \Cref{L:ConsistencyAc} (consistency of $\sdda \interpt u$) along with the bound $
		|\lambda_j| \leq |u|_{W^2_{\infty}(B_i)}
		$ on the eigenvalues $\lambda_j$ of $D^2u$.
	\end{proof}
        
	\begin{remark}[improved consistency error] \label{R:OptConsAc} We observe that for a smooth function $u \in C^{5,1}(\overline \Omega)$ \Cref{L:FullConsistencyAc} gives a consistency error of order $h^2$, instead of only order $h$, which corresponds again to the optimal choice $\da = h^{1/2}$.
	\end{remark}
	
	\section{Filtered Scheme - Analysis} \label{S:FilteredScheme}
In this section we prove the existence of discrete solutions of \eqref{E:FiltOp} and their convergence to the unique viscosity solution of \eqref{E:MA}, the two main theoretical results of this paper. An important component for this analysis is the following property of the filtered operator $\Tef$, which mimics \Cref{L:Monotonicity} (monotonicity of $\Tem$) and hinges on the uniform bound of the filter $F$.
	\begin{Lemma}[almost monotonicity of $\Tef$] \label{L:AlmostMon}
		For two grid functions $\mathbf{V}_h,\mathbf{W}_h$ such that $\mathbf{V}_h-\mathbf{W}_h$ attains a maximum at an interior node $x_i \in \Nhi$, we have that
		$$
		\Tef[\mathbf{W}_h](x_i) +2 \tau \geq \Tef[\mathbf{V}_h] (x_i).
		$$
	\end{Lemma}
	\begin{proof}
		Let $v_h^1,w_h^1 \in \Vh^1$ and $v_h^2,w_h^2 \in \Vth^2$ be the corresponding functions with the same nodal values as $\mathbf{V}_h$ and $\mathbf{W}_h$, respectively.
		We use the definition \eqref{E:FiltOp} of $\Tef$ and property $|F(s)| \leq 1$ for all $s$ to write
		$$
		\Tef[\mathbf{W}_h](x_i) = \Tem[w_h^1](x_i) + k_i \  \tau, \quad \Tef[\mathbf{V}_h](x_i) = \Tem[v_h^1](x_i) + \tilde k_i \ \tau
		$$
		where $|k_i|,|\tilde k_i| \leq 1$.
		We now apply \cite[Lemma 2.3 (monotonicity)]{NoNtZh} to $\Tem$ and $v_h^1,w_h^1$, whose difference attains a maximum at $x_i$, to deduce that
		$$
		\Tem[w_h^1](x_i)  \geq \Tem[v_h^1](x_i).
		$$
		Combining these expressions, we obtain
		$$
		\Tef[\mathbf{W}_h](x_i) + (\tilde k_i - k_i) \tau \geq \Tef[\mathbf{V}_h] (x_i),
		$$
		whence the assertion follows immediately.
	\end{proof}

	\subsection{Existence of Discrete Solution} \label{S:FiltEx}
	In this section we prove that \eqref{E:FiltOp} has a discrete solution $\Uve$. This hinges on \Cref{L:AlmostMon} (almost monotonicity of $\Tef$) and the existence results in \cite{NoNtZh} for $\Tem$. In fact, we combine the latter with a fixed point argument as in \cite{FrOb2}. Moreover, we show that, although we cannot guarantee uniqueness, we can control the $l^\infty(\Nh)$ difference between distinct discrete solutions.
        
	\begin{Lemma} [existence and stability]\label{L:Existence}
          Let the filter $F_\sigma$ be defined by either \eqref{E:Filter} if $f>0$
        or \eqref{E:NSFilter} if $f\ge 0$. Then, there exists a grid function $\Uve$ that solves \eqref{E:FiltOp} and so that the corresponding function $\uve^1$ is discretely convex. Moreover, $\Uve$ is stable in the sense that $\|\Uve\|_{l^{\infty}(\Nh)}$
		does not depend on the parameters involved in $\epsilon$.
	\end{Lemma}
	\begin{proof}
		We first show that for any grid function $\mathbf{U}_h$, with corresponding functions $u_h^1 \in \Vh^1$ discretely convex and $u_h^2 \in \Vth^2$, we can find a grid function $\mathbf{Y}_h(\mathbf{U}_h)$ with corresponding function $y_h^1=y_h^1(\mathbf{U}_h^1) \in \Vh^1$ such that for all $x_i \in \Nhi$
		\begin{equation} \label{E:fixedpointeqn}
	\Tem[y_h^1(\mathbf{U}_h)](x_i) = f(x_i) - \tau F_\sigma\Big(\frac{A[\mathbf{U}_h](x_i)}{\tau} \Big).
		\end{equation}
                We proceed in two steps.

		{\it Step 1: existence of $y_h^1$.} Let $\mathbf{U}_h$ be a grid function with corresponding function $u_h^1 \in \Vh^1$ discretely convex, or equivalently $\Tem[u_h^1](x_i)\ge0$ for all $x_i\in\Nhi$. If $f\ge f_0>0$, then for $0<\tau\le f_0$ we infer that
                  \[
                  f(x_i) - \tau F_\sigma\Big(\frac{A[\mathbf{U}_h](x_i)}{\tau} \Big) \ge 0\quad\forall \, x_i\in\Nhi.
                  \]
                  On the other hand, if $f\ge0$ then \eqref{E:NSFilter} implies $F_\sigma(s)\le0$ for all $s$ and the above inequality holds again. We next extend $F_\sigma\big(\tau^{-1}A[\mathbf{U}_h](x_i)\big)$ as a continuous piecewise linear function to $\Omega_h$ and apply the existence result for $\Tem$ from \cite[Lemma 3.1]{NoNtZh} to conclude that there exists a unique solution $y_h^1(\mathbf{U}_h) \in \Vh^1$ to \eqref{E:fixedpointeqn} with $\|y_h^1(\mathbf{U}_h)\|_{L^{\infty}} \leq \Lambda$ and $\Lambda$ depends on $\|g\|_{L^{\infty}(\partial\Omega)}$ and $\|f\|_{L^{\infty}(\Omega)}$. We have thus constructed a grid function $\mathbf{Y}_h(\mathbf{U}_h)$ with nodal values given by $y_h^1(\mathbf{U}_h)$ and such that $\|\mathbf{Y}_h(\mathbf{U}_h)\|_{\ell^\infty(\Nh)} \le \Lambda$.

\smallskip
	        {\it Step 2: fixed point argument.} Since $F_\sigma$ is continuous for any $\sigma>0$, and the solution of \eqref{E:fixedpointeqn} depends continuously on data in $L^\infty(\Omega_h)$, according to \cite[Proposition 4.6]{NoNtZh2}, we deduce that the map $\mathbf{U}_h \mapsto \mathbf{Y}_h(\mathbf{U}_h)$ is continuous. In addition, the set of grid functions $\mathbf{U}_h$ with corresponding discretely convex $u_h^1 \in \Vh^1$ that satisfy both the boundary condition $\mathbf{U}_h(x_i)=g(x_i)$ for all $x_i\in\Nhb$ as well as the uniform bound $\|\mathbf{U}_h\|_{l^{\infty}(\Nh)} \leq \Lambda$, is compact and convex. Since $\mathbf{U}_h \mapsto \mathbf{Y}_h(\mathbf{U}_h)$ maps this set into itself, we can apply the Brouwer's fixed point theorem to find $\mathbf{U}_h$ such that $\mathbf{Y}_h(\mathbf{U}_h)=\mathbf{U}_h$ which is thus a solution to \eqref{E:FiltOp}. This concludes the proof.
\end{proof}

        \begin{remark}[non-uniqueness] We emphasize that the above proof does not guarantee the existence of a unique solution to \eqref{E:FiltOp}, since in principle we can have more than one fixed points for \eqref{E:fixedpointeqn}. However, the next lemma shows that two different solutions of \eqref{E:FiltOp} are very close to each other. Their distance in the $l^\infty$ norm is dictated by the filter scale $\tau$.

\end{remark}	
\begin{Lemma}[control of the lack of uniqueness] Let $\Uve, \Vve$ be two discrete solutions of \eqref{E:FiltOp} with corresponding discretely convex functions $\uve^1,\vve^1\in\Vh^1$. Then,
	$$
\|\Uve-\Vve\|_{l^\infty(\Omega)} \leq C \ \tau^{1/d},
	$$
	where $C$ depends only on the dimension $d$ and $\Omega$.
\end{Lemma}
\begin{proof}
  The result is an immediate consequence of \cite[Lemma 2.4 (discrete comparison principle)]{NoNtZh} for $\Tem$ and the fact that \eqref{E:discrete-convex} implies
  \[
  \Tem[\uve^1](x_i) \geq 0,~ \Tem[\vve^1](x_i) \geq 0
  \quad\forall \, x_i \in \Nhi.
  \]
In fact, we use the discrete barrier $q_h = \interpo (|x-x_0|^2-R^2) \in \Vh^1$, introduced in \cite[Lemma 5.2]{NoNtZh}, where $x_0$ and $R>0$ are such that $\Omega \subset B_R(x_0)$. Since $\Tef[\Uve](x_i) = \Tef[\Vve](x_i) = f(x_i)$ for all $x_i \in \Nhi$, we have that
	$$
	\Tem[\uve^1] (x_i)\leq \Tem[\vve^1](x_i) + 2 \tau \leq \Tem\Big[\vve^1+\frac{(2\tau)^{1/d}}{2}q_h\Big](x_i)
	$$
	and
	$$
		\Tem[\vve^1](x_i) \leq \Tem[\uve^1](x_i) + 2 \tau \leq \Tem\Big[\uve^1+\frac{(2\tau)^{1/d}}{2}q_h\Big](x_i)
		$$
		for all $x_i \in \Nhi$. Applying \cite[Lemma 2.4]{NoNtZh}, we obtain
		$$
		\|\uve^1-\vve^1\|_{L^\infty(\Omega_h)} \leq \frac{(2\tau)^{1/d}}{2} \|q_h\|_{L^\infty(\Omega_h)} \leq C \tau^{1/d},
		$$
		which concludes the proof.
	\end{proof}

\subsection{Convergence} \label{S:Convergence}
	We now prove convergence of the function $\uve^1$ associated with the solution $\Uve$ of \eqref{E:FiltOp} to the unique viscosity $u$ solution of \eqref{E:MA}. To this end, we follow the proof of convergence of \cite[Theorem 5.7]{NoNtZh}, which in turn modifies that of \cite{BaSoug} to account for the Dirichlet boundary conditions and the lack of operator consistency near $\partial \Omega_h$. In addition, we exploit the almost monotone nature of the scheme, as in \cite{FrOb2}, to further adjust the proof of \cite{BaSoug} and derive uniform convergence of $\uve^1$ to $u$ in $\Omega$. As in \cite{NoNtZh}, we also resort to a discrete barrier argument to control the behavior of the discrete solution close to the boundary. We start with the discrete barrier function from \cite[Lemma 5.1]{NoNtZh}.
	\begin{Lemma}[discrete boundary barrier] \label{L:Barrier} Let $\Omega$ be uniformly convex and  $E>0$ be arbitrary. For each node $z \in \Nhi$ with ${\rm dist}(z,\partial \Omega_h) \leq \delta$, there exists a function $p_h\in\Vh^1$ such that $\Tem[p_h](x_i) \geq E$ for all $x_i \in \Nhi$, $p_h \leq 0$ on $\partial \Omega_h$ and
		\[
		|p_h(z)| \leq CE^{1/d} \delta,
		\]
		with $C$ depending on the curvature of the boundary.
	\end{Lemma}
	Before proceeding further, we recall the following continuous version of the \MA operator 
	\begin{equation*}
	T[u] := \min_{\bv=(v_j)_{j=1}^d\in\Vperp} \left( \prod_{j=1}^d \partial^{2,+}_{v_jv_j} u
	- \sum_{j=1}^d \partial^{2,-}_{v_jv_j}u \right),
	\end{equation*}  
	where $\partial^{2,+}_{v_jv_j} u := \max\big( \partial^2_{v_jv_j} u,0\big)$ and $\partial^{2,-}_{v_jv_j} u := -\min\big( \partial^2_{v_jv_j} u,0\big)$. The following equivalence between convex viscosity solutions of \eqref{E:MA} and viscosity solutions of
	\begin{equation} \label{E:MAFull}
	T[u] = f \quad\textrm{in }\Omega,
	\qquad
	u = g \quad\textrm{on }\partial\Omega,
	\end{equation}
	is proven in \cite[Lemma 5.6]{NoNtZh}.
          
	\begin{Lemma}[equivalence of viscosity solutions] \label{L:EquivalenceVisc}
		If $f\in C(\Omega)$ satisfies $f \geq 0$, and $u\in C(\overline{\Omega})$,
		then $u$ is a
		viscosity solution of \eqref{E:MAFull} if and only if $u$ is a convex
		viscosity solution of \eqref{E:MA}.
	\end{Lemma}
        
	We are now in a position to prove the uniform convergence of $\uve^1\in\Vh^1$ in $\Omega$. Since $\uve^1$ is defined in the computational domain $\Omega_h$,
		and $\Omega_h\subset\Omega$, we extend $\uve^1$ to $\Omega$ as
		follows. Given $x\in\Omega\setminus \Omega_h$ let
		$z\in\partial\Omega_h$ be the closest point to $x$, which is unique
		because $\Omega_h$ is convex, and let
		\begin{equation}\label{extension}
		\uve^1 (x) := \uve^1(z) = \interpo g(z)
		\quad\forall \, x \in \Omega\setminus\Omega_h.
		\end{equation}
	This will allow control of the behavior of $\uve^1$ close to $\partial\Omega$ using techniques from \cite{NoNtZh}.
	
	We also introduce the limit supremum and the limit infimum of $\uve^1$, namely
	$$
	u^*(x) = \limsup_{\varepsilon \to 0, z \to x} \uve^1(z), 
	\qquad
	u_*(x) = \liminf_{\varepsilon \to 0, z \to x} \uve^1(z)
	\quad \forall x \in \Omega,
	$$
	where we require without explicit statement that $\frac{h}{\dm}, \frac{h}{\da} \to 0$ as $h \to 0$. We observe that $u^*$ is upper semi-continuous and $u_*$ is lower semi-continuous. Since the proof follows closely the one in \cite{NoNtZh}, we emphasize only the parts of it that are different for the filtered operator. In the following calculations we do not rely on the precise definition of the filter function $F_\sigma$ but use \Cref{L:AlmostMon} (almost monotonicity of $\Tef$).
	
	\begin{Theorem}[uniform convergence] \label{T:Convergence}
		Let $\Omega$ be uniformly convex, $f\in C(\overline{\Omega}) \cap L^\infty(\Omega)$
		satisfy $f\ge 0$, and $g\in C(\partial\Omega)$.
		The function $\uve^1\in\Vh^1$ of (\ref{E:FiltOp}) converges uniformly to the unique 
		viscosity solution $u\in C(\overline{\Omega})$ of \eqref{E:MA}
		as $\ve = \ve(h) \rightarrow 0$.
	\end{Theorem}
	\begin{proof}
		In view of \Cref{L:EquivalenceVisc} (equivalence of viscosity solutions),
		we prove instead that $\uve^1$ converges to the viscosity solution $u$ of \eqref{E:MAFull} uniformly.
		To this end, we have to deal with a test function $\phi\in C^2(\Omega)$
		and the respective grid function $\mathbf{\Phi}_h$ with corresponding piecewise polynomial functions $\phi_h^1 = \interpo \phi \in \Vh^1$ and $\phi_h^2 = \interpt \phi \in \Vth^2$. Without loss of
		generality we may assume $\phi \in C^{2,\alpha}(\Omega)$.
		We split the proof into five steps. 
		
		\smallskip  
		\textit{Step 1: Consistency.}
                Let $x_0\in\Omega$ and $x_i\in\Nhi\cap\Omega_{h,\delta_m}$.
		We have the following consistency estimate for the operator $T$ in \eqref{E:MAFull}, which is an immediate consequence of \Cref{L:FullConsistencyMon} (consistency of $\Tem[\interpo u]$), the Lipschitz continuity of the min and max functions, and the fact that $|\Tef[\mathbf{\Phi}_h](x_i)-\Tem[\phi_h^1](x_i)|\leq \tau$:
		$$
		\big| T[\phi](x_0) - \Tef[\mathbf{\Phi}_h](x_i)  \big|
		\le C_1(\phi) \Big(\delta_m^\alpha + |x_0-x_i|^\alpha \Big) +
		C_2(\phi) \Big( \frac{h^2}{\delta_m^2} + \theta_m^2 \Big) + \tau.
		$$
		Here the constants $C_1, C_2$ are defined in \Cref{L:FullConsistencyMon}
		and depend on
		$|\phi|_{C^{2,\alpha}(B_i)}$ and $|\phi|_{W^2_\infty(B_i)}$ with $B_i$
		defined in \eqref{E:Bi}.
		
		\smallskip  
		\textit{Step 2: Subsolutions.} We show that $u^*$ is a
		viscosity subsolution of \eqref{E:MAFull};
		likewise $u_*$ is a viscosity supersolution.
		This hinges on monotonicity and consistency \cite{BaSoug}. In our case, we employ \Cref{L:Monotonicity} (monotonicity of $\Tem$) and \Cref{L:AlmostMon} (almost monotonicity of $\Tef$). We must show that if $u^* - \phi$ attains a local maximum at $x_0\in\Omega$, we have
		$$
		T[\phi](x_0) \geq f(x_0);
		$$
		note that $u^*-\phi$ is upper semi-continuous and the local maximum
		is well defined.
		Without loss of generality, we may assume that $u^* - \phi$ attains a
		strict global maximum at $x_0$ \cite[Remark in p.31]{IL} and $x_0 \in \Omega_h$ for $h$ sufficiently small.
		Let $x_h\in\mathcal{N}_h$ be a sequence of nodes so that $\Uve - \mathbf{\Phi}_h$
		attains a maximum at $x_h$. We claim that, as in \cite{NoNtZh}, $x_h \to x_0$ as $h \to 0$.
		Exploiting the fact that $\Uve - \mathbf{\Phi}_h$ attains a maximum at $x_h$,
		\Cref{L:AlmostMon} (almost monotonicity of $\Tef$) yields
		$$
		\Tef[\mathbf{\Phi}_h](x_h) + 2\tau \geq \Tef[\uve](x_h) = f(x_h),
		$$
		where $\tau \to 0$ as $\ve \to 0$.
		Since $f\in C(\overline{\Omega})$, to prove $T[\phi](x_0) \ge f(x_0)$                
		we only need to show that as $\varepsilon \to 0$
		$$
		\Tef[\mathbf{\Phi}_h](x_h) \to T[\phi](x_0).
		$$
		This is a consequence of Step 1 and the fact that
		$x_h\in\Omega_{h,\dm}$ for $\dm$ sufficiently small, because $x_0\in\Omega$, $x_h \to x_0$
			and the sequence of $\Omega_h\uparrow\Omega$ is non-decreasing.
		
		\smallskip
		\textit{Step 3: Boundary Behavior.}
		We now prove that $u^*=u_*=g$ on $\partial\Omega$ via a barrier
		argument similar to those in \cite{FeJe,NoNtZh,NoZh};
		we proceed as in \cite{FeJe}. This is essential  in order to apply the comparison principle for operator $T$ to relate	$u_*, u^*$ and $u$ in Step 4.

		Let $p_k$ be the quadratic function in the proof of
		\Cref{L:Barrier} (discrete boundary barrier)
		associated with an arbitrary boundary point $x\in\partial\Omega$
		(the origin in the construction of $p_k$) and with constant $E=k$. We
		recall that $p_k(x) =0$ and $p_k(z) \leq 0$ for all $z \in \partial\Omega$
		can be made arbitrarily large for $k\to\infty$ by virtue of the
		uniform convexity of $\Omega$. A simple consequence is that the
		sequence of points $x_k\in\partial\Omega$ where $g+p_k$
		(resp. $g-p_k$) attains a maximum (resp. a minimum)
		over $\partial\Omega$ converges to $x$ as $k \to \infty$.
		
		\Cref{L:Monotonicity} (monotonicity of $\Tem$) implies the following maximum principle for the 
		monotone operator:
		if $\Tem[u_h^1](x_i)>0$ for all $x_i\in\Nhi$ is valid for a discretely convex function $u_h^1 \in \Vh^1$,	then $u_h^1$ attains a maximum over $\Omega_h$ on $\Nhb \subset \partial\Omega$.
		We now see that, since $\Tem[\uve^1] \geq 0$, we have that for $k$ big enough and all $x_i \in \Nhi$
			$$
			\begin{aligned}
			\Tem[\uve^1 + \interpo p_k](x_i)  &\geq \Tem[\uve^1]+\Tem[\interpo p_k](x_i) \\
			& \geq \Tef[\Uve](x_i) - \tau + E  \\ 
			&= f(x_i) -\tau +E >0,
			\end{aligned}
			$$
		whence $\uve^1 + \interpo p_k$
		attains its maximum on $\Nhb$. In view of \eqref{extension}, we
			may assume $z\in\Omega_h$ in the limit $u^*(x) = \limsup_{\varepsilon, \frac{h}{\delta} \to 0, z \to x}\uve^1(z)$. Consequently,
		\begin{align*}
		u^*(x) & \le \limsup_{\varepsilon \to 0, z \to x}
		\big(\uve^1(z)+\interpo p_k(z) \big) - \liminf_{\varepsilon
			\to 0, z \to x} \interpo p_k(z) \\
		& \le   \limsup_{\varepsilon \to 0} \
			\max_{z \in \Nhb }  \big(g + p_k\big) (z) -  p_k(x) 
		\le g(x_k) + p_k(x_k) \le g(x_k),
		\end{align*}
		because $\max_{\Nhb} g+p_k \le \max_{\partial \Omega } g + p_k$. Hence taking $k \to \infty$ yields  $u^*(x) \leq g(x)$.
		
		On the other hand, since $\Tem[\interpo p_k](x_i) > \Tem[\uve^1](x_i)$
		for all $x_i\in\Nhi$ and $k$ big enough, Lemma \ref{L:Monotonicity} (monotonicity of $\Tem$) implies that
		$\uve^1-\interpo p_k$ attains a minimum on $\Nhb$. Therefore,
		arguing as before
		\begin{align*}
		u_*(x)\ge \liminf_{\varepsilon \to 0}
			\min_{z \in \partial{\Omega}} \big(g - p_k\big)(z) +  p_k(x)
		\ge g(x_k) - p_k(x_k) \ge g(x_k),
		\end{align*}
		whence $u_*(x) \geq g(x)$. This in turn gives
		$u^* \leq g \leq u_* \leq u^*$ on $\partial\Omega$
		as asserted.

		\smallskip
		\textit{Step 4: Comparison.} To prove that $u^*=u_*$ in
		$\overline{\Omega}$ we make use of the comparison principle in \cite{NoNtZh} for \eqref{E:MAFull}. Since $u^*$ and $u_*$ are a subsolution and supersolution respectively of \eqref{E:MAFull} and they agree on the boundary, we can deduce that $u^*\le u_*$ in $\overline{\Omega}$. Combining with $u^* \ge u_*$, by
		definition, this results in $u^*=u_*$ in $\overline{\Omega}$.
		
		\smallskip 
		\textit{Step 5: Uniform Convergence.} This is identical to \cite{NoNtZh} and is thus omitted. The proof is complete.
	\end{proof}

	\section{Conclusions}
	In this paper we introduce two methods to solve the \MA equation \eqref{E:MA}. The first one is an accurate scheme that hinges on quadratic interpolation and a higher-order approximation of directional derivatives in \eqref{E:Det}. It exhibits errors in the $L^\infty$ norm of one to two orders of magnitude lower than the monotone operator introduced in \cite{NoNtZh}. However, formal higher order accuracy comes at the cost of monotonicity, which prevents us from proving convergence in $L^\infty$ for this operator. The second method circumvents this issue by combining the monotone and the accurate operators into a filtered scheme. This yields convergence to the viscosity solution relying on stability and monotonicity properties of the monotone operator and the fact that the filter scale $\tau\to0$ as $h\to0$. We employ two filter functions according to whether the forcing $f$ is strictly positive or degenerate. In both cases, the discrete piecewise linear solution $\uve^1\in\Vh^1$ is discretely convex. The filter detects parts of the domain where the accurate operator could under-perform due to lack of regularity of the solution, as it happens in our degenerate example, and switches to the monotone operator. We explore the two methods computationally and illustrate the enhanced performance of both schemes by comparing them with the numerical experiments from \cite{NoNtZh}. Lastly, we investigate the effect of filter function and filter scale and discuss some computational challenges of the method.
	
	\medskip\noindent
	{\bf Acknowledgments.}
	We are indebted to S.W. Walker for providing assistance and guidance
	with the software FELICITY and to H. Antil for numerous discussions
	about the implementation of the method. We also thank W. Zhang for early discussions about the filter methodology.

	\bibliographystyle{amsplain}

\end{document}